\documentclass[11pt,reqno]{amsart}
\usepackage{amssymb}
\usepackage{mathptmx}
\usepackage{mathtools}
\usepackage{mathrsfs}
\usepackage[all]{xy}
\usepackage{stmaryrd}
\usepackage{fancyhdr}
\usepackage{hyperref}
\usepackage{mathrsfs}
\usepackage{tikz-cd}
\usepackage{cite}
\usepackage{amsmath,amsfonts}
\usepackage{graphicx}
\usepackage{wrapfig}
\usepackage{array}
\usepackage{amsthm}
\usepackage[left=1.2in, right=1.2in, top=1in, bottom=1in]{geometry}
\usepackage{etoolbox}
\usepackage{hyperref}
\patchcmd{\section}{\scshape}{\bfseries}{}{}
\makeatletter
\renewcommand{\@secnumfont}{\bfseries}
\makeatother
\xyoption{all}
\DeclareMathOperator{\Spec}{Spec}
\DeclareMathOperator{\Spv}{Spv}
\DeclareMathOperator{\Frac}{Frac}

\DeclareMathOperator{\Ker}{Ker}

\newcommand{\angles}[1]{\langle #1 \rangle}

\input xy
\xyoption{all}
\theoremstyle{definition}
\newtheorem{mydef}{\textbf{Definition}}[section]
\newtheorem{myeg}[mydef]{\textbf{Example}}

\newtheorem{rmk}[mydef]{\textbf{Remark}}
\newtheorem*{que}{\textbf{Question}}

\theoremstyle{plain}
\newtheorem{mythm}[mydef]{\textbf{Theorem}}

\newtheorem*{nothma}{\textbf{Theorem A}}
\newtheorem*{nothmb}{\textbf{Theorem B}}
\newtheorem*{nothmc}{\textbf{Theorem C}}
\newtheorem*{nothmd}{\textbf{Theorem D}}

\newtheorem{mytheorem}[mydef]{\textbf{Theorem}}
\newtheorem{lem}[mydef]{\textbf{Lemma}}
\newtheorem{pro}[mydef]{\textbf{Proposition}}

\newtheorem{cor}[mydef]{\textbf{Corollary}}

\patchcmd{\abstract}{\scshape\abstractname}{\normalsize{\textbf{\abstractname}}}{}{}
\begin{document}

\title{Lattices, Spectral Spaces, and Closure Operations on Idempotent Semirings}
\author{Jaiung Jun}
\address{Department of Mathematics, State University of New York at New Paltz, New Paltz, NY 12561, USA}
\email{junj@newpaltz.edu}

\author{Samarpita Ray}
\address{Department of Mathematics, Indian Institute of Science Education and Research, Pune 411008, India}
\email{ray.samarpita31@gmail.com}

\author{Jeffrey Tolliver}
\address{Fortstone Research 18th floor 50 Milk St Boston, MA 02110, USA}
\email{jeff.tolli@gmail.com}

\subjclass[2010]{16Y60~(primary), 14T05~(primary), 06F30~(secondary), 06D99~(secondary), 13A18~(secondary), 54D80~(secondary) }
\keywords{semiring, idempotent semiring, spectral space, coherent space, space of valuations, closure operation, bounded distributive lattice, congruence, $k$-ideal, prime congruence, integral closure, valuation order, Frobenius closure}
\thanks{}

\begin{abstract}

Spectral spaces, introduced by Hochster, are topological spaces homeomorphic to the prime spectra of commutative rings. In this paper we study spectral spaces in perspective of idempotent semirings which are algebraic structures receiving a lot of attention due to its several applications to tropical geometry. We first prove that a space is spectral if and only if it is the \emph{ prime $k$-spectrum} of an idempotent semiring.  In fact, we enrich Hochster's theorem by constructing a subcategory of idempotent semirings which is antiequivalent to the category of spectral spaces. We further provide examples of spectral spaces arising from sets of congruence relations of semirings. In particular, we prove that the \emph{space of valuations} and the \emph{space of prime congruences} on an idempotent semiring are spectral, and there is a natural bijection of sets between the two; this shows a stark difference between rings and idempotent semirings. We then develop several aspects of commutative algebra of semirings. We mainly focus on the notion of \emph{closure operations} for semirings, and provide several examples. In particular, we introduce an \emph{integral closure operation} and a \emph{Frobenius closure operation} for idempotent semirings. 
\end{abstract}

\maketitle

%\tableofcontents

%\vspace*{6pt} % for this guide only.
% A table of contents should normally not be included

\section{Introduction}

A semiring is an algebraic structure which assumes the same axioms as a ring except that one does not necessarily require additive inverses to exist. Typical examples include the semiring $\mathbb{N}$ of natural numbers, the Boolean semiring $\mathbb{B}$, or the tropical semiring $\mathbb{T}$. The theory of semirings has its own charms, but also recently people have found several applications of semirings, making the theory of semirings even more interesting. 

One application arises from tropical geometry. Tropical geometry is a new branch of algebraic geometry, where one studies an algebraic variety by means of its combinatorial shadow (a polyhedral complex) obtained from the underlying set of an algebraic variety and a valuation on a ground field. As commutative rings provide an algebraic foundation of algebraic geometry, commutative semirings provide an algebraic foundation of tropical geometry. The question of which algebraic structure provides ``the best'' algebraic foundation for tropical geometry has not been settled yet. There are various candidates, for instance, hyperfields \cite{viro}, blueprints \cite{lorscheid2019tropical} and tropical schemes \cite{giansiracusa2016equations} (or more generally, tropical ideals \cite{maclagan2016tropical}), which is the first paper introducing scheme theory to tropical geometry. However, the theory of idempotent semirings~\footnote{Recall that by an idempotent semiring, we mean a semiring such that $a+a=a$ for any element $a$.} is in the intersection of several approaches towards the algebraic foundation for tropical geometry, and hence it is important to develop tools and ideas for the commutative algebra of idempotent semirings to this end.

Another motivation to study commutative algebra of semirings arises from A.~Connes and C.~Consani's program on developing algebraic geometry in ``characteristic one'', where one develops basic languages and tools for algebraic geometry over more general algebraic structures (than commutative rings) to shed some light on the Riemann hypothesis; one fundamental idea of Connes-Consani program is to translate Weil's proof of the Riemann hypothesis for algebraic curves to the case of $\Spec \mathbb{Z}$. To this end, one should be able to understand $\Spec \mathbb{Z}$ as a ``curve'' defined over some field-like object (typically called ``the field with one element'' $\mathbb{F}_1$), and hence one should work beyond the category of commutative rings as $\mathbb{Z}$ is the initial object in the category of commutative rings and hence $\mathbb{Z}$ cannot be understood as an ``$\mathbb{F}_1$-algebra'' in the category of commutative rings. 

One essential ingredient of Connes-Consani program, which should be developed, is the language of homological algebra in characteristic one, which is very far from working with abelian categories. Semirings in this case seem to provide a reasonable algebraic structure on which homological algebra in characteristic one can be built as it was shown in \cite{connes2017homological} by Connes and Consani. Also, for an approach which simultaneously deals with homological algebra for semirings and hyperfields, we refer the reader to \cite{jun2019projective}.
	
Related, but slightly diverged motivation arises from J.~Borger's work. In \cite{borger2016boolean} and \cite{borger2016witt}, Borger proved that the big Witt functor can be generalized to semirings by observing that the big Witt functor is representable by the ring of symmetric functions, which has an $\mathbb{N}$-basis. In a similar vein, Borger claims that algebraic geometry over the semiring $\mathbb{N}$ encodes certain positivity of algebraic geometry over $\mathbb{Z}$ in a suitable way. See, also \cite{culling} for the notion of the \'{E}tale fundamental group of a scheme over $\mathbb{N}$ along with interesting examples. 

In both tropical geometry and $\mathbb{F}_1$-geometry, the theory of \emph{idempotent semirings} plays a key role. For instance, a tropical variety can be seen as the common zeros of polynomials with coefficients in an idempotent semiring. See \cite{maclagan2009introduction}. Also, in Connes and Consani's approach to $\mathbb{F}_1$-geometry, their main objects to develop the theory are idempotent semirings. See, for example, \cite{con5,connes2017homological}. One of the main motivations for the current paper is to contribute to the aforementioned momentum by developing and bringing more tools to the commutative algebra of idempotent semirings. In particular, we study spectral spaces, valuations, and closure operations in the context of idempotent semirings. We remark that we explore closure operations for general semirings, and find some interesting examples in the case of idempotent semirings. To the best of our knowledge, closure operations for semirings have never been considered, whereas valuations of semirings have been intensively studied, for instance \cite{jun2018valuations},  \cite{giansiracusa2016equations}, \cite{izhakian2011supertropical}.

%For a given commutative ring $A$, there is a one-to-one correspondence between ideals and congruence relations. However, for a semiring $A$, we do not have this correspondence anymore in general. One obtains a congruence relation from an ideal, however a congruence relation does not uniquely define an ideal in general. So, the theory of ideals and the theory of congruence relations diverge for semirings. Therefore, in the current paper, we study closure operations and spectral spaces arising from semirings in perspectives of ideals and congruence relations simultaneously. We also note that in \cite{ray2018closure}, the second author studied closure operations and valuations on monoids and spectral spaces arising in these contexts, which could be potentially related to the current paper. 

We first search for an analogous statement of Hochster's theorem on prime spectra and spectral spaces for the case of idempotent semirings. In his seminal work \cite{hochster1969prime}, Hochster provides a topological characterization of prime spectra of commutative rings by introducing the notion of a spectral space, which is a quasi-compact, $T_0$, and sober topological space such that the set of all quasi-compact open subsets is an open basis and a finite intersection fo quasi-compact open subsets is again a quasi-compact open subset. In the case of semirings, one can easily show that the prime spectrum of a semiring is spectral by using the exact same argument as rings. So, one may ask for a given spectral space $X$, whether or not we can find a semiring $A$ in such a way that the prime spectrum of $A$ is homeomorphic to $X$. However, as rings are semirings, this is just a tautology. We instead ask the following question:

\begin{que}
For a given spectral space $X$, can we find an idempotent semiring (which can never be a ring) $A$ in such a way that the prime spectrum of $A$ is homeomorphic to $X$?
\end{que}

To answer this question, one cannot simply mimic Hochster's proof since many of Hochster's constructions fail to hold for the case of idempotent semirings. Our strategy to answer the above question is to appeal to the well-known relation between spectral spaces and bounded distributive lattices. In fact, we prove that the answer is affirmative if we restrict ourselves to a specific class of ideals, called \emph{ $k$-ideals} (Definition \ref{definition: saturated ideal}).  One of the first discussions on $k$-ideals of semirings can be found in \cite{senadhikari1992}. We define the {\it prime $k$-spectrum} of a semiring $A$ as the collection of all prime $k$-ideals of $A$ endowed with the hull-kernel topology (Definition \ref{k-spec}). We denote the {\it prime $k$-spectrum} by $Spec_k(A)$ and we prove the following:

\begin{nothma}(Theorem \ref{theorem: spectral theorem})
Let $X$ be a spectral space. Then, there exists an idempotent semiring $A$ such that the prime $k$-spectrum $\Spec_k A$ of $A$ is homeomorphic to $X$. 
\end{nothma}
Also, the prime $k$-spectrum of any semiring is a spectral space (see, Proposition \ref{sat prime spec}(4)). Therefore, it follows from Theorem \ref{theorem: spectral theorem} that a space is spectral if and only if it is the prime $k$-spectrum of an idempotent semiring.\\
We mention here that Lescot introduced a notion of saturated ideals of a semiring in \cite{les2} and proved that when the semiring is idempotent, an ideal $I$ is a $k$-ideal if and only if $I$ is a saturated ideal (see, \cite[\S 8]{PP}). In Corollary 6.3 of \cite{les3}, Lescot has shown that the saturated prime spectrum of a semiring is a spectral space. Thus, when the semiring is idempotent, this implies that the prime $k$-spectrum is a spectral space. However, the proof of our Proposition \ref{sat prime spec}(4)) is different (and shorter) and it holds for any semiring.\\
 \\
In proving Theorem A, we also prove that there is a nice categorical equivalence between the opposite category of the category of spectral spaces and the category of certain semirings. We remark that our categorical equivalence is an enrichment of Hochster's theorem in the following sense; Hochster's theorem provides an essential image of the prime spectrum functor, however, his construction does not give an equivalence of a subcategory of commutative rings with the opposite category of spectral spaces. On the other hand, we prove that one can find a subcategory of idempotent semirings which is equivalent to the opposite category of spectral spaces as follows:

\begin{nothmb}(Theorem \ref{theorem: equivalence of cat})
There is an equivalence of categories between radical idealic semirings~\footnote{An idempotent semiring is equipped with a canonical order, and by an idealic semiring, we mean an idempotent semiring such that every element is bounded by $1$ with respect to the canonical order. See Definition \ref{definition: idealic}. It is further said to be radical if $x^2=x$ for all $x$. }(as a subcategory of the category of semirings) and bounded distributive lattices.  Furthermore, this equivalence commutes with forgetful functors (i.e. it is the identity on the level of sets). In particular, as the category of bounded distributive lattices is antiequivalent to the category of spectral spaces, we can conclude that the category of radical idealic semirings is also antiequivalent to the category of spectral spaces.
\end{nothmb}

Next, we study the space of valuations on an idempotent semiring $A$ in connection with the space of valuation orders on $A$ introduced by the third author in \cite{tolliver2016extension}. We also provide a natural bijection (as sets) between the space of valuations and the space of prime congruences as in \cite{bertram2017tropical} and \cite{joo2014prime}.  We remark that our bijection is very specific to idempotent semirings which is not true for rings, see Remark \ref{remark: rem}. We prove the following:

\begin{nothmc}(Propositions \ref{pro: Spv and Spec}, \ref{pro: spectral1})
Let $A$ be an idempotent semiring. Let $\Spv A$ (resp. $\Spec_\mathfrak{c} A$) the space of valuations (resp. the space of prime congruences) on $A$. Then, there is a natural bijection of sets between $\Spv A$ and $\Spec_\mathfrak{c} A$. Furthermore, $\Spv A$ and $\Spec_\mathfrak{c} A$ are spectral spaces. 
\end{nothmc}

%Finally, we turn to the notion of closure operations on idempotent semirings. 
Recall that a closure operation on ideals (or modules in general) of a commutative ring extends a given ideal in a certain way (depending on each closure operation). Intuitively speaking in geometric pictures, a closure operation may clear away some ``bad part'' from a closed subscheme; for instance, when it comes to curves, the integral closure operation ``clears away singular points''. Closure operators were also studied thoroughly by Andrew Dudzik \cite{dudzik2017quantales} in the context of hyperstructures.

For a commutative ring $A$, there is a bijective correspondence between ideals and congruence relations (see the proof of Corollary 2 of \cite{j2}). However, for a semiring $A$, we do not have this bijection anymore in general. %One obtains a congruence relation from an ideal, however a congruence relation does not uniquely define an ideal in general. 
So, the theory of ideals and the theory of congruence relations diverge for semirings. To this end, in Section \ref{section: section3}, we consider closure operations on both ideals and congruence relations of a semiring and obtain spectral spaces arising from finite type closure operations (analogous to the case of rings). We also note that in \cite{ray2018closure}, the second author studied closure operations and valuations on monoids and spectral spaces arising in these contexts, which could be potentially related to the current paper. 

Our main construction will be an integral closure operation on an idempotent semiring in Section \ref{section: closure operations for semirings}, paralleling the integral closure of the case of rings. To be a bit more specific, let $A$ be an idempotent semiring and $I$ be an ideal of $A$. We define the following set:
\[
I^{{int}}:=\{x\in A~ \mid ~x^n+ a_1x^{n-1} + \hdots + a_n = b_1x^{n-1} + \hdots + b_n \text{ for some } n \in \mathbb{N} \text{ and } a_i,b_i\in I^{i}\}.
\]
We also let $I'$ be the intersection of all  $k$-ideals of $A$ containing $I$. With this, we prove the following among other things. 

\begin{nothmd}(Proposition \ref{pro: integral closure})
Let $A$ be an idempotent semiring and $I$ be an ideal of $A$. Then, $I^{{int}}$, which denotes the set of integral elements over the ideal $I$, is an ideal of $A$. Furthermore, 
\[
I \mapsto (I^{int})',
\]
where $(I^{int})'$ is the $k$-closure of the ideal $I^{int}$, defines a closure operation on the set $\mathcal{I}$ of all ideals of $A$, which we call \emph{the integral closure operation}. 
\end{nothmd}

Finally, we introduce the Frobenius closure for idempotent semirings in $\S \ref{subsection: frobenius closure}$. We also interpret the radical operation for congruences, first introduced in \cite{bertram2017tropical} and further studied in \cite{joo2014prime}, as a closure operation on the set of congruences on $A$. \\

This paper is organized as follows. In $\S2$, we provide the necessary terminology for the paper. In $\S 3$, we prove Theorem A and an enrichment (Theorem B) providing a categorical equivalence between a subcategory of semirings and the category of spectral spaces. In the same section, we also consider closure operations on both ideals and congruence relations of a semiring and obtain spectral spaces arising from finite type closure operations. In $\S 4$, we introduce the space of valuations and valuation orders for idempotent semirings, and prove that they are spectral spaces. We further provide a natural bijection between the space of valuations and the set of prime congruences. Finally, in $\S 5$, we explore several examples of closure operations for idempotent semirings. In particular, we introduce the notion of an integral closure and a Frobenius closure in this setting. 
\vspace{0.5cm}

\textbf{Acknowledgments} J.J. was supported by AMS-Simons travel grant. J.J. and S.R. thank Kalina Mincheva and D\'aniel Jo\'{o} for helpful conversations. J.J. and J.T. are grateful to the organizers of the JAMI workshop: Riemann-Roch in characteristic one and related topics at the Johns Hopkins University, where they could work on the project together. The authors also thank to Neil Epstein for helpful comments on the first draft of the paper. 

\section{Preliminaries}

In this section, we introduce the basic objects studied in this paper. In the first subsection, we recall basic definitions and examples of semirings and congruences. In the second subsection, we review the notion of spectral spaces and provide some examples. In the third subsection, we briefly recall the definition of closure operation. The readers, who are familiar with the topics in this section, may skip to the next section.

\subsection{Semirings, Congruences, and Lattices}
In this section, we review basic definitions for semirings, congruences, and lattices which will be used in this paper. We assume that all semirings are commutative, unless otherwise stated, starting from the following definition. 

\begin{mydef}
By a \emph{semiring}, we mean a nonempty set $A$ with two binary operation $+$ and $\cdot$ such that $(A,+)$ and $(A,\cdot)$ are commutative monoids and $(a+b)c= ac+bc$ and $0\cdot a=0$ for all $a,b,c \in A$.  When $(A-\{0\},\cdot)$ is a group, $A$ is said to be a \emph{semifield}. 
\end{mydef}

A semiring $A$ is said to be \emph{additively idempotent} if $a+a=a$ for all $a \in A$.

\begin{myeg}(Boolean semifield)
Let $\mathbb{B}:=\{0,1\}$. One defines multiplication as usual. Addition is defined as follows:
\[
0+0=0, \quad 1+0=1, \quad 1+1=1. 
\]	
$\mathbb{B}$ is called the \emph{Boolean semifield}. 
\end{myeg}

\begin{myeg}(Tropical semifield)
Let $\mathbb{T}:=\mathbb{R} \cup \{-\infty\}$. One defines multiplication of $\mathbb{T}$ as the ordinary addition of $\mathbb{R}$ with the rule that $a\cdot (-\infty)=-\infty$ for any $a \in \mathbb{T}$. Addition is defined as follows:
\[
x+y:=\max\{x,y\}.
\]
We have $0_\mathbb{T}=-\infty$ and $1_\mathbb{T}=0$. $\mathbb{T}$ is called the \emph{tropical semifield}. 
\end{myeg}

\begin{mydef}
Let $A$ and $B$ be semirings. A \emph{homomorphism} from $A$ to $B$ is a function $f:A \to B$ such that for all $a,b \in A$, 
\[
f(a+b)=f(a)+f(b), \quad f(ab)=f(a)f(b), \quad f(0)=0, \quad f(1)=1.
\]
\end{mydef}

Here are some examples. 

\begin{myeg}
The following function
\[
f: \mathbb{T} \longrightarrow \mathbb{B}, \quad a \mapsto \begin{cases}
1 \textrm{ if $a \neq -\infty$ } \\
0 \textrm{ if $a = -\infty$}
\end{cases}
\]
is a homomorphism. Also the following function
\[
g: \mathbb{B} \longrightarrow \mathbb{T}, \quad f(0)=-\infty, \quad f(1)=0.
\]
is a homomorphism. 
\end{myeg}

\begin{mydef}
Let $A$ be a semiring. 
\begin{enumerate}
\item 
$A$ is said to be an \emph{integral} semiring if $A$ does not have any zero divisor, i.e., $ab=0$ implies either $a=0$ or $b=0$ $\forall a,b \in A$.
\item
$A$ is said to be a \emph{multiplicatively cancellative} semiring if $A$ satisfies the following:
\[
ac=bc \implies a=b, \quad \forall a,b \in A \textrm{ and }c \neq 0 \in A.
\]
%\item $A$ is said to be an \emph{additively cancellative} semiring if $A$ satisfies the following:
%\[
%a+c=b+c \implies a=b, \quad \forall a,b,c\in A.
%\]
\item 
$A$ is said to be \emph{strict} (or {\it zero-sum free}) if for each $a \in A$, there is no $x \in A$ such that $x+a=0$. 
\end{enumerate}
\end{mydef}

\begin{rmk}
If $A$ is a commutative ring, then being integral is same as being multiplicatively cancellative. However, when $A$ is a semiring, cancellativity implies integrality in general, but not conversely. For instance, the polynomial semiring $\mathbb{T}[x]$ with coefficients in the tropical semifield $\mathbb{T}$ is integral but not cancellative. 
\end{rmk}

From now on, by an idempotent semiring we always mean an additively idempotent semiring unless otherwise stated.

\begin{mydef}\label{definition: saturated ideal}
Let $A$ be a semiring.
\begin{enumerate}
\item 
By an \emph{ideal} of $A$, we mean an additive submonoid $I$ such that $AI \subseteq I$.
\item
An ideal $I$ which is not $A$ is called a \emph{proper ideal}.
\item 
A {\it $k$-ideal} of $A$ is an ideal $I$ which satisfies the following condition:
 $$ x \in I,~x+y \in I\implies y \in I \qquad \forall~x,y \in A$$
\item
A \emph{prime ideal} of $A$ is a proper ideal $\mathfrak{p}$ of $A$ such that if $ab \in \mathfrak{p}$ then $a \in \mathfrak{p}$ or $b \in \mathfrak{p}$ $\forall a,b \in A$. 
\item
A \emph{maximal ideal} of $A$ is a proper ideal $\mathfrak{m}$ which is not contained in any other proper ideal. 
\end{enumerate}
\end{mydef}

Let $A$ be a semiring. As in the classical case, let $X=\Spec A$ be the set of prime ideals of $A$ and we impose topology on $X$ in such a way that the closed sets are of the following form:
\[V(I):=\{\mathfrak{p} \in X \mid I \subseteq \mathfrak{p}\}.\]
One can mimic the classical construction of a structure sheaf for a semiring spectrum $\Spec A$ to make $\Spec A$ a locally semiringed space. In general, a semiring scheme is defined to be a locally semiringed space which is locally isomorphic to $\Spec A$ for some semiring $A$. For details, we refer the readers to \cite{jun2017vcech}.

Next, we recall the definition of a \emph{congruence} $C$ on a semiring $A$. By a congruence $C$ on $A$, we mean an equivalence relation on $A$ which is compatible with the algebraic structure of $A$, i.e., for any $a,b,c,d \in A$, if $a \sim b$ and $c\sim d$, then we have $a+c \sim b+d$ and $ac \sim bd$. Equivalently, a congruence $C$ is a subsemiring of $A \times A$ which is an equivalence relation. To be specific, a congruence is a subset $C$ of $A \times A$ satisfying the following conditions: for any $a,b,c,d \in A$,
\begin{enumerate}
	\item 
	$(a,a) \in C$; (reflexive) 
	\item 
	$(a,b) \in C$ if and only if $(b,a) \in C$; (symmetric).
	\item 
	If $(a,b),(b,c) \in C$, then $(a,c) \in C$; (transitive).	
	\item 
	If $(a,b),(c,d) \in C$, then $(a+c,b+d) \in C$; (compatible with addition). 
	\item 
	If $(a,b),(c,d) \in C$, then $(ac,bd) \in C$; (compatible with multiplication). 	
\end{enumerate}

\begin{mydef}
Let $A$ be a semiring and $C_1$ and $C_2$ be congruences on $A$. We write $C_1 \subseteq C_2$ if $C_1$ is a subset of $C_2$ by considering $C_1$ and $C_2$ as subsets of $A \times A$. We say that $C_1$ is a \emph{subcongruence} of $C_2$ if $C_1 \subseteq C_2$. 
\end{mydef}

\begin{rmk}\label{remark: right before}
In terms of equivalence relations, $C_1$ being a subcongruence of $C_2$ means that that if $a \sim_{C_1} b$ then $a \sim_{C_2} b$. 	
\end{rmk}

Recall that by a congruence $C$ on $A$ generated by a set $X \subseteq A \times A$, we mean the following congruence:
\[
C:=\bigcap_{X \subseteq E} E,
\]
that is, the intersection of all congruences $E$ containing $X$. 
The recipe to construct $C$ is as follows:
\begin{enumerate}	
	\item 
	Construct $X':=X \cup \{(a,b) \in A \times A : (b,a) \in X\}$. 
	\item 
	Construct $X'':=X' \cup \{(a,a): a \in A\}$. 
	\item 
	Construct the subsemiring $C_0$ of $A \times A$ generated by $X''$. 
	\item 
	Take the transitive closure $\overline{C_0}$, then $C=\overline{C_0}$. 
\end{enumerate}

The point of this construction is that after taking the transitive closure, we do not have to go back to $(1)$, i.e., $C=\overline{C_0}$. 

In \cite{bertram2017tropical}, A.~Bertram and R.~Easton first introduced (and further studied by Jo\'o and Mincheva in \cite{joo2014prime}) the \emph{twisted product} $x \cdot_t y$ of elements $x=(x_1,x_2),y=(y_1,y_2) \in A \times A$ as follows:
\[
(x\cdot_t y):=(x_1y_1+x_2y_2,x_1y_2+x_2y_1). 
\]

The product\footnote{One may also use the coordinate-wise product to define the product of two congruences which produces a different congruence.} of two congruences $C$ and $D$ is defined as the congruence generated by the following set:
\[
\{c \cdot_t d~|~ c\in C,~d\in D\}.
\] 
The twisted product has been introduced to find a ``good'' notion of prime congruences in tropical geometry. In fact, Jo\'{o} and Mincheva used the twisted product to prove a version of the tropical nullstellensatz. We will prove in $\S 4$ that there is a natural bijection between the set of prime congruences (prime congruence spectrum) and the space of valuations.

\begin{rmk}\label{remark: classical theory}
In ring theory, there is one-to-one correspondence between ideals and congruence relations (see the proof of Corollary 2 of \cite{j2}). 
%If $I_i$ are ideals and $C_i$ are corresponding congruence relations, we have that
%\[
%C_1 \subseteq C_2 \iff I_1 \subseteq I_2. 
%\]
One stark contrast between ring theory and semiring theory is that there is no longer one-to-one correspondence between ideals and congruence relations for semirings. Furthermore, even if $C$ and $D$ are finitely generated congruences on a semiring $A$, $CD$, as the congruence generated by the set $\{c \cdot_t d~|~ c\in C,~d\in D\}$, does not have to be finitely generated as Example \ref{example: finitely generated} illustrates. 
\end{rmk}

\begin{myeg}\label{example: finitely generated}
The authors learned this example from D\'aniel Jo\'o. This example shows that even when two congruences $C$ and $D$ are generated by a single element, $CD$ may not be finitely generated. 

Let $A=\mathbb{B}[x,y]$ be the two variable polynomial semiring with coefficients in the Boolean semifield $\mathbb{B}$. Let $C$ and $D$ be the congruences generated by $(x,y)$. Then, the congruence $CD$ will have the elements of the following form:
\[
(x^{2n} + y^{2n},~ x^n y^n), \quad  n \in \mathbb{N}. 
\]
But, one can easily see by induction that in $CD$ there is no non-diagonal element of the form $(x^k,t)$ nor $(y^k,q)$ for $t\neq x^k,q \neq y^k \in \mathbb{B}[x,y]$. This is because $(x,y)\cdot_t(x,y)=(x^2+y^2,xy)$ and the congruence $CD$ is obtained by the procedure described right after Remark \ref{remark: right before}; so, it is not possible to have a monomial $x^k$ nor $y^k$. Also, one may notice that $x^{2n}+y^{2n}$ does not factor over $\mathbb{B}$. It follows that no non-diagonal pair of the form $(x^{2n}+y^{2n},\dots)$ can be generated by adding or multiplying lower degree relations. In particular, $(x^{2n}+y^{2n},x^ny^n)$ cannot be in the transitive closure of what we can obtain by adding and multiplying lower degree relations. Therefore, $CD$ is not finitely generated since we have to add at least $(x^{2n}+y^{2n},x^ny^n)$ for each $n \in \mathbb{N}$ in a set of generators. 
\end{myeg}

Finally, we recall some definitions of the lattice theory, which will be used to prove an idempotent analogue of Hochster's theorem. 
 
\begin{mydef}
Let $(L,\preceq)$ be a partially ordered set. 
\begin{enumerate}
\item 
$(L,\preceq)$ is said to be a \emph{lattice} if for any $x,y \in L$, the greatest lower bound $x\wedge y$ and the least upper bound $x \vee y$ exist. 
\item 
A lattice $(L,\preceq)$ is said to be \emph{distributive} if for any $x,y,z \in L$, we have
\[
x \vee (y \wedge z) = (x \vee y) \wedge (x \vee z). 
\]
\item 
A lattice $(L,\preceq)$ is said to be \emph{bounded} if there exist elements $0, 1 \in L$ such that for any $x \in L$, $0 \preceq x \preceq 1$, or equivalently $x \vee 1 = 1$ and $x \wedge 0=0$. 
\item 
A lattice $(L,\preceq)$ is said to be \emph{complete} if for any subset $M \subseteq L$, $\vee_{x \in M} x$ and $\wedge_{x \in M} x$ exist. 
\end{enumerate}	
\end{mydef}

\begin{myeg}
Any totally ordered set $(L,\preceq)$ is a distributive lattice. 
\end{myeg}

\begin{mydef}
\noindent	\begin{itemize}
		\item[(1)] Let $L_1$ and $L_2$ be lattices. A {\it lattice homomorphism} $f:L_1 \to L_2$ is a function such that for any $x,y \in L_1$,
		\[
		f(x \vee y) = f(x) \vee f(y), \quad f(x \wedge y) = f(x) \wedge f(y). 
		\]
		\item[(2)] Let $L_1$ and $L_2$ be bounded lattices. Then, a lattice homorphism $f:L_1 \to L_2$ is said to be a {\it $\{0,1\}$-lattice homomorphism} if $f(0)=0$ and $f(1)=1$. 
	\end{itemize}

\end{mydef}

Note that the bounded distributive lattices with the $\{0,1\}$-lattice homomorphisms form a category and this is what we mean by the {\it category of bounded distributive lattices} in this paper.

\subsection{Spectral spaces}
In \cite{hochster1969prime}, Hochster provided purely topological characterization of affine schemes by introducing the notion of spectral spaces. We first recall the definition. 

\begin{mydef}\cite{hochster1969prime}
A \emph{spectral space} is a quasi-compact and $T_0$ topological space $X$ which satisfies the following conditions:
\begin{enumerate}
\item 
A finite intersection of quasi-compact open subsets of $X$ is again a quasi-compact open subset of $X$ and the set of all quasi-compact open subsets of $X$ forms a basis of $X$. 
\item
Any nonempty irreducible closed subspace $Y$ of $X$ has a unique generic point, i.e., $\exists~ \eta \in Y$ such that $\overline{\{\eta\}}=Y$. 
\end{enumerate}

For a commutative ring $A$, it is clear that $\Spec A$ is, as a topological space, a spectral space. Hochster proved that for a given spectral space $X$, there exists a commutative ring $A$ such that $\Spec A$ is homeomorphic to $X$. Hochster further proved that this construction is functorial. 

\end{mydef}

 Recently, Finocchiaro developed a new criterion involving {\it ultrafilters} to characterize spectral spaces \cite{F}. We first recall the definition of ultrafilters before stating the criterion.

\begin{mydef}\label{filter}
A nonempty collection $\mathcal{F}$ of subsets
of a given set $X$ is called a \emph{filter} on $X$ if the following properties hold\footnote{Strictly speaking, the definition of filters that we provide here is commonly called proper filters.}:
\begin{enumerate}
\item 
$\emptyset \notin \mathcal{F}$.	
\item 
If $Y,Z \in \mathcal{F}$, then $Y\cap Z \in \mathcal{F}$.
\item 
If $Z \in \mathcal{F}$ and $Z\subseteq Y \subseteq X$, then $Y \in \mathcal{F}$.
\end{enumerate}

A filter $\mathcal{F}$ is said to be an \emph{ultrafilter} if for each $Y\subseteq X$, either $Y \in \mathcal{F}$ or $X \setminus Y \in \mathcal{F}$. We shall denote an ultrafilter by $\mathcal{U}$.
\end{mydef}

For further details and examples of filters, see, for instance, \cite{Je}.

The following two results provide a way to produce spectral spaces by using ultrafilters.

\begin{mythm}\cite[Corollary 3.3]{F}\label{2.2}
Let $X$ be a topological space. The following conditions are equivalent:
\begin{enumerate}
\item 
$X$ is a spectral space.
\item 
$X$ satisfies the $T_0$-axiom and there is a subbasis $\mathbb{S}$ of $X$ such that 
\begin{equation*}
X_{\mathbb{S}}({\mathcal{U}}):= \{x \in X\ |\ [\forall S\in \mathbb{S}, x \in S \iff S \in \mathcal{U}]\} \neq \emptyset
\end{equation*}
for any ultrafilter $\mathcal{U}$ on $X$.
\end{enumerate}
\end{mythm}

\subsection{Closure operations}

In this subsection, we first recall the notion of closure operation for any partially ordered set. We then recall some classical examples of closure operations arising in commutative algebra from \cite{epstein2012guide}.

\begin{mydef}\label{clo-op}
	Let $P$ be a partially ordered set. A {\it closure operation} on $P$ is a map $\textrm{cl}:P\longrightarrow P$ such that
%	Let $A$ be a set and let $\mathcal{S}\subseteq 2^A$. A \emph{closure operation} \textrm{cl} on $\mathcal{S}$ is a function $\textrm{cl}: \mathcal{S} \to \mathcal{S}$ such that 
	\begin{enumerate}
		\item
		(Extension) $ x\leq cl(x)$ for all $x\in P$. 
		\item 
		(Idempotence) $\textrm{cl}(x)=\textrm{cl}(\textrm{cl}(x))$ for all $x \in P$. 
		\item
		(Order-preservation) $x\leq y\implies \textrm{cl}(x) \leq \textrm{cl}(y)$ for all $x, y \in P$. 
	\end{enumerate} 
\end{mydef}

\iffalse
\begin{mydef}
, and. A \emph{closure operation} \textrm{cl} is a function $\textrm{cl}: \mathcal{I} \to \mathcal{I}$ such that 
\begin{enumerate}
	\item
	(Extension) $I \subseteq \textrm{cl}(I)$ for all $I \in \mathcal{I}$. 
	\item 
	(Idempotence) $\textrm{cl}(I)=\textrm{cl}(\textrm{cl}(I))$ for all $I \in \mathcal{I}$. 
	\item
	(Order-preservation) If $I_1$ and $I_2$ are ideals in $\mathcal{I}$ such that $I_1 \subseteq I_2$, then $\textrm{cl}(I_1) \subseteq \textrm{cl}(I_2)$. 
\end{enumerate} 
\end{mydef}
\fi
Let $A$ be a commutative ring and let $\mathcal{I}\subseteq 2^A$ be the set of all ideals of $A$. Clearly, $\mathcal{I}$ is partially ordered by inclusion. Here are some examples of closure operations on $\mathcal{I}$. For notational convenience, we let $\textrm{cl}(I):=I^\textrm{cl}$ for $I\in \mathcal{I}$.

\begin{myeg}\cite[Example 2.1.2]{epstein2012guide}
Let $\mathcal{I}$ be the set of all ideals of a commutative ring $A$.  
\begin{enumerate}
\item 
The function $\textrm{cl}:\mathcal{I} \to \mathcal{I}$, sending $I$ to $I^\textrm{cl}:=I$, is a closure operation, called the \emph{identity closure}. 
\item 
The function $\textrm{cl}:\mathcal{I} \to \mathcal{I}$, sending $I$ to $I^\textrm{cl}:=\sqrt{I}$ is a closure operation, called the \emph{radical closure}. 
\item 
For an ideal $I$ of $A$, an element $a \in A$ is said to be integral over $I$, if there exist $n \in \mathbb{N}$ and $a_i \in I^i$ for $i=1,2,...,n$ such that 
\[
a^n+a_1a^{n-1}+a_2a^{n-2}+\cdots + a_{n-1}a + a_n =0.
\]
We let $I^\textrm{cl}$ be the set of elements in $A$ which are integral over $I$. Then, the function $\textrm{cl}:\mathcal{I} \to \mathcal{I}$, sending $I$ to $I^\textrm{cl}$ is a closure operation, called the \emph{integral closure}. 
\item 
Suppose that $A$ is of characteristic $p >0$. For an ideal $I$, we define $I^{[p^n]}$ to be the ideal generated by all the $p^n$th power of elements of $I$. Then,
\[
I^\textrm{cl}:= \{x\in A~\mid~\exists~ n \in \mathbb{N} \text{~such that~} x^{p^n} \in  I^{[p^n]}\}
\]
defines a closure operation, called the \emph{Frobenius closure}. 
\end{enumerate}
\end{myeg}

We will provide interesting examples of closure operations arising from idempotent semirings in \S \ref{section: closure operations for semirings}. In particular, we will introduce an integral closure operation and Frobenius closure operation for idempotent semirings, which could be seen as ``characteristic one'' analogues of the case of rings. We also would like to highlight the fact that one can construct a closure operation of finite type from a given closure operation (see, \cite[Construction 3.1.6]{epstein2012guide}), and a closure operation of finite type naturally produces spectral spaces in our case. See, Proposition \ref{proposition: of fnite type}.

\section{Spectral spaces arising from semirings}\label{section: section3}

In this section, we explore spectral spaces arising from semirings. In the first subsection, we present spectral spaces obtained from certain collections of ideals of a semiring. We then prove that any spectral space is homeomorphic to the \emph{prime $k$-spectrum} (Definition \ref{k-spec}) of an idempotent semiring (we provide explicit construction). This could be viewed as an idempotent semiring analogue of Hochster's theorem. However, our proof does not  mimic Hochster's proof since many of Hochster's constructions fail to hold for the case of idempotent semirings. Our strategy appeals to the well known relation between spectral spaces and bounded distributive lattices. In fact, in the process of our proof, we enrich Hochster’s theorem by constructing a subcategory of idempotent semirings (\emph{radical idealic semirings}) which is antiequivalent to the category of spectral spaces. In the second subsection, we illustrate several spectral spaces constructed from sets of congruences on a semiring. \vspace{2mm}

In \cite{hochster1969prime}, Hochster introduced the notion of patch topology for a spectral space $X$:  if $X$ is a spectral space with topology $\tau$, then the patch topology on $X$ is the topology $\tau'$ whose sub-basis for closed sets are the closed sets and quasi-compact open sets of $(X,\tau)$. Hochster proved that any patch closed subset of a spectral space is spectral \cite[Proposition 9]{hochster1969prime}. We will frequently use this fact to prove sets of ideals (or congruences) of an idempotent semiring are spectral spaces by realizing our space of interest as a patch closed subset of a spectral space. 

\subsection{Spectral spaces arising from ideals}

Recall the definition of a $k$-ideal of a semiring $A$ (Definition \ref{definition: saturated ideal}). Equivalently, an ideal $I$ of $A$ is a $k$-ideal if $x+y\in I$ implies either $x,y \in I$ or $x,y\notin I$.
We can explicitly describe the smallest $k$-ideal containing a given ideal by the following proposition.\footnote{The proof of Proposition \ref{saturation} is essentially same as that of \cite[Lemma 2.2]{senadhikari1992}, where the authors prove the same result for specific semirings.}

\begin{pro}\label{saturation}
	Let $A$ be a semiring and let $I$ be an ideal of $A$.  Then the smallest $k$-ideal containing $I$ is $\{x\in A\mid\exists y\in I\textrm{ such that }x+y\in I\}$.
\end{pro}
\begin{proof}
	Let $J=\{x\in A\mid\exists y\in I\textrm{ such that }x+y\in I\}$.  It is clear that $I\subseteq J$ and for any $k$-ideal  $N$ with $I\subseteq N$ we have $J\subseteq N$.  It is also clear that $J$ is closed under multiplication.  Let $x, x'\in J$.  Then there exist $y, y'\in I$ such that $x+y, x'+y'\in I$.  Then $y+y'\in I$ and $(x+x')+(y+y')\in I$ and so $x+x'\in J$.  Thus $J$ is an ideal.  Let $x\in J$ and suppose $z\in A$ satisfies $x+z\in J$.  Then there exist $y, w$ such that $y, x+y, w, x+z+w\in I$.  Consequently, $x+y+w, x+y+z+w\in I$ and so $z\in J$. This shows that $J$ is a $k$-ideal.
\end{proof}
%If $x\in A$, we will often use $\langle x \rangle$ to denote the smallest $k$-ideal containing $x$.

\begin{mydef}\label{addmul}
	Let $A$ be a semiring, and $I$, $J$ be $k$-ideals of $A$. 
	\begin{enumerate}
		\item 
		The sum $I+J$ of $k$-ideals is the smallest $k$-ideal containing $I$ and $J$.
		\item 
		The product of $k$-ideals $IJ$ is the smallest $k$-ideal containing $\{xy \mid x\in I, y\in J\}$.
	\end{enumerate}	
\end{mydef}

\begin{rmk}
	We remark that Definition \ref{addmul} is well defined as the intersection of arbitrary $k$-ideals is a $k$-ideal.
\end{rmk}

Let $A$ be a semiring and $\mathcal{I}$ be the set of all ideals of $A$. One can easily observe that $\mathcal{I}$ is equipped with the semiring structure, where addition is the sum $I+J$ of two ideals, and multiplication is the product $IJ$ of two ideals. We note that the sum or product of two ideals depends on whether we view them as $k$-ideals or just ideals.  In other words, the semiring of $k$-ideals is not a subsemiring of the semiring of all ideals\footnote{However, it is the quotient obtained by identifying two ideals $I$ and $J$ if the smallest $k$-ideal containing $I$ is the same as the smallest $k$-ideal containing $J$.}

\begin{mydef}
	Let $A$ be a semiring
	\begin{enumerate}
		\item 
		A \emph{prime $k$-ideal} of $A$ is a $k$-ideal $\mathfrak{p}$ of $A$ such that if $k$-ideals $I,J$ satisfy $IJ\subseteq \mathfrak{p}$, then $I\subseteq \mathfrak{p}$ or $J\subseteq \mathfrak{p}$ (here the product $IJ$ is the product of $k$-ideals).
		\item 
		The \emph{radical} $\sqrt{I}$ of a $k$-ideal $I$ is the intersection of all the prime $k$-ideals containing $I$.
	\end{enumerate}
\end{mydef}

\noindent If $A$ is a ring, then any ideal of $A$ is a $k$-ideal, and hence prime $k$-ideals are the same as prime ideals for a ring $A$. Also, a radical ideal of a ring $A$ is a $k$-ideal by being an intersection of prime ideals (and hence $k$-ideals). We will now show that for a semiring $A$, prime $k$-ideals are the same as prime ideals which are also $k$-ideals.

\begin{pro}
	Let $A$ be a semiring and $\mathfrak{p}$ be an ideal of $A$. Then, $\mathfrak{p}$ is a prime $k$-ideal of $A$ if and only if $\mathfrak{p}$ is a prime ideal of $A$ which is also a $k$-ideal.
	
	%For $x,y \in A$, if $xy \in \mathfrak{p}$, then $x\in \mathfrak{p}$ or $y \in \mathfrak{p}$. 
\end{pro}
\begin{proof}
	For any $x\in A$, let $\angles{x}$ denote the smallest $k$-ideal containing $x$. We claim that $\angles{xy}=\angles{x}\angles{y}$ for any $x,y\in A$. First, if $z \in \angles{x}$, then $zy \in \angles{xy}$.  In fact, there exist $a,b \in A$ such that $ax+z=bx$ (by Proposition \ref{saturation}) and so $axy+zy=bxy$. Thus, $zy\in \angles{xy}$ for any $z\in \angles{x}$. Now pick $w\in \angles{y}$ and choose $c,d\in A$ such that $cy+w=dy$. Then, $cyz+wz=dyz$ and so $wz\in \angles{xy}$ (since $yz\in \angles{xy}$).  This establishes the claim. \\
	Now, let $\mathfrak{p}$ be a prime $k$-ideal of $A$.  For $x,y \in A$, if $xy \in \mathfrak{p}$, then $\angles{xy}=\angles{x}\angles{y}\subseteq \mathfrak{p}$ and so either $\angles{x}\subseteq \mathfrak{p}$ or $\angles{y}\subseteq \mathfrak{p}$. In other words, we have $x \in \mathfrak{p}$ or $y \in \mathfrak{p}$. \\
	Conversely, suppose $\mathfrak{p}$ is a prime ideal of $A$ which is also a $k$-ideal. Let $I$ and $J$ be $k$-ideals such that $IJ\subseteq \mathfrak{p}$. If $I\nsubseteq \mathfrak{p}$, then there exists some $x\in I$ such that $x\notin \mathfrak{p}$. But, by assumption, $xy\in \mathfrak{p}$ for all $y\in J$. Since $\mathfrak{p}$ is a prime ideal, we must have $J\subseteq \mathfrak{p}$. 
\end{proof}

We will now present certain collections of ideals of a semiring that form spectral spaces. For this, we first recall that given a set $S$, the power set $2^S$ is a spectral space endowed with the  \emph{hull-kernel topology} whose open sub-basis is given by the sets of the form $$D(F):= \{I \in 2^S \mid F \not \subseteq I\},$$ where $F$ is a finite subset of $S$ (see, \cite[Theorem 8 and Proposition 9]{hochster1969prime}). We will denote by $V(F)$ the complement of $D(F)$ in $2^S$.

\begin{pro}\label{sat prime spec}
Let $A$ be a semiring. 
\begin{enumerate}
\item 
The collection of all ideals and the collection of all proper ideals of $A$ are spectral spaces with the hull-kernel topology.
\item 
The collection of all prime ideals $\Spec A$ of $A$ is a spectral space with the hull-kernel topology.
\item
The collection of all $k$-ideals of $A$ is a spectral space with the hull-kernel topology.
\item
The collection of all prime $k$-ideals of $A$ is a spectral space with the hull-kernel topology.
\end{enumerate}
\end{pro}
\begin{proof}
The power set $2^A$ of $A$ is a spectral space with the hull-kernel topology. Let $X_1$ denote the collection of all ideals of $A$. Clearly, $X_1$ is a subset of $2^A$. Endowed with the induced hull-kernel topology $X_1$ can be written as follows
\[
X_1=[\underset{a,b\in A}\cap D(a)\cup V(ab)] \cap [\underset{a,b\in A}\cap D(a) \cup D(b) \cup V(a+b)] \cap V(0). 
\]
This shows that $X_1$ is patch closed in $2^A$ and therefore it is a spectral space. Let $X_2$ denote the collection of all proper ideals. Then, we have 
\[
X_2= X_1\cap D(1)
\]
which is clearly a patch closed subset of $2^A$ and is therefore spectral. 
Let $X_3$ and $X_4$ denote the collection of all prime ideals and the collection of all $k$-ideals respectively. Then, we have

\[
X_3=X_2\cap [\underset{a,b\in A}\cap D(ab)\cup V(a)\cup V(b)]
\]

\[
X_4=X_2\cap[\underset{a,b\in A\times A}\cap D(a+b)\cup [V(a)\cap V(b)]\cup [D(a)\cap D(b)]]
\]
which are patch closed subsets of $2^A$ and are therefore spectral. The collection of all prime $k$-ideals is given by the patch closed subset
$X_3\cap X_4$. 
\end{proof}

\begin{mydef}\label{k-spec}
	The {\it prime $k$-spectrum} of a semiring $A$ is the collection of all prime $k$-ideals of $A$ endowed with the hull-kernel topology and we denote it by $Spec_k(A)$.
\end{mydef}

Let $A$ be a semiring and $\mathcal{I}$ be the poset of all ideals of $A$. We now consider closure operations on $\mathcal{I}$ and present spectral spaces related to that.

\begin{mydef}\label{cl op on ideals}
Let $A$ be a semiring and $\mathcal{I}$ be the poset of all ideals of $A$. A closure operation 
\[
cl: \mathcal{I}\longrightarrow \mathcal{I}\qquad I\mapsto I^{cl}
\]	
is said to be {\it of finite type} if 
	$I^{cl} = \bigcup \{J^{cl}\ |\  J\subseteq I,\  J \in \mathcal{I},\  J  \text{  is finitely generated } \}.$
	
\end{mydef}

We will discuss many interesting examples of finite type closure operations on ideals of an idempotent semiring in Section \ref{section: closure operations for semirings}. We now show that the fixed points of any finite type closure operation on $\mathcal{I}$ gives rise to a spectral space. The proof of this is analogues to the proof of \cite[Proposition 3.4]{Fo}, but we include it here for completeness.

\begin{pro}\label{proposition: fnite type and spectral space}
Let $A$ be a semiring and $cl$ be a closure operation of finite type on $\mathcal{I}$, the poset of all ideals of $A$. Then the following set
	\[
	X:=\{I \in \mathcal{I} \mid I^{cl}=I\}
	\]
	is a spectral space. 
\end{pro}
\begin{proof}
Let $\mathcal{U}$ be an ultrafilter on $X$ and $\mathbb{S}$ be the subbasis of $X$ which is induced from the hull-kernel topology of $2^{A}$. By Theorem \ref{2.2}, it is enough to prove that
\begin{equation*}
X_{\mathbb{S}}({\mathcal{U}})= \{I \in X\ |\ [\forall S\in \mathbb{S}, I \in S \iff S \in \mathcal{U}]\} \neq \emptyset.
\end{equation*}
Consider the set $I_{\mathcal{U}}:=\{x \in I \mid V(x)\cap X  \in \mathcal{U}\}$. We claim that $I_{\mathcal{U}} \in X_{\mathbb{S}}({\mathcal{U}})$. Using the properties of an ultrafilter, it can be easily verified that $I_{\mathcal{U}}$ satisfies the condition $I_{\mathcal{U}} \in S \iff S \in \mathcal{U}$ for all $S\in \mathbb{S}$. So, we are only left to check that $I_{\mathcal{U}}^{cl}\subseteq I_{\mathcal{U}}$. Suppose that $x \in I_{\mathcal{U}}^{cl}$. Since our closure operation $cl$ is of finite type, there is a finitely generated ideal $I' \subseteq I_{\mathcal{U}}$ such that $x\in (I')^{cl}$. It follows that $x \in J^{cl}$ for any ideal $J$ of $A$ containing $I'$. Hence, if $I'$ is finitely generated by $\{a_1,\hdots ,a_r\}$, we have that
\begin{equation}\label{inclusion1}
\cap_{i=1}^r V(a_i)\cap X = V(a_1,\hdots ,a_r) \cap X \subseteq V(x) \cap X. 
\end{equation}
By definition of $I_{\mathcal{U}}$, we have $ V(a_i)\cap X\in \mathcal{U}$ and since $\mathcal{U}$ is an ultrafilter, it follows from \eqref{inclusion1} that $x\in I_{\mathcal{U}}$.
\end{proof}

\begin{rmk}
Closure operations can also be considered for the poset of subsemimodules of a given semimodule $M$ over $A$. Proposition \ref{proposition: fnite type and spectral space} will also hold (proof is exactly the same) for the poset of subsemimodules of a given semimodule in place of poset of all ideals of $A$.
\end{rmk}

Inspired by Hochster's result, we now proceed to prove that for a spectral space $X$, one can find an idempotent semiring $A$ in such a way that the prime $k$-spectrum of $A$ is homeomorphic to $X$. We will appeal to the well-known fact in lattice theory that there exists one-to-one correspondence between spectral spaces and bounded distributive lattices to prove this. 
%We first recall some definitions and basic properties of lattice theory. Also, to avoid confusion about notation, we briefly review the theory of $k$-ideals of an idempotent semiring.

\subsubsection{$k$-ideals of idempotent semirings }

Let $A$ be an idempotent semiring. Then $A$ is equipped with a canonical partial order as follows: for $x,y \in A$, 
\[
x \leq y \iff x+y = y
\]
In particular, with this partial order, $0$ is the smallest element of $A$. Furthermore, this order is compatible with the multiplication of $A$, that is, if $x\leq y$ then $xz\leq yz$ for any $z \in A$. As a consequence, one can easily show that $x \leq y$ and $a \leq b$ imply that $ax \leq by$. In the case of idempotent semirings, it is well-known that $k$-ideals have the following simple description.

\begin{pro}
Let $A$ be an idempotent semiring. Then, an ideal $I$ of $A$ is a $k$-ideal if and only if for all $x\in I$ and $y\leq x$ it follows that $y\in I$.
\end{pro}
\begin{proof}
Suppose for all $x\in I$ and $y\leq x$ it follows that $y\in I$.  Let $y\in I$ and $x+y\in I$.  Then $x\leq x+y$ and so $x\in I$. Conversely, suppose $I$ is a $k$-ideal.  Let $x\in I$ and $y\leq x$.  Then, $x+y=x\in I$ and so $y\in I$.
\end{proof}

\begin{mydef}
A $k$-ideal $I\subseteq A$ is \emph{finitely generated} if it is the smallest $k$-ideal containing some finite subset of $A$.  If $x\in A$, we will often use $\langle x \rangle$ to denote the smallest $k$-ideal containing $x$. If $x$ is only an element of an idempotent semigroup, it will denote the smallest $k$-subsemigroup\footnote{By a $k$-subsemigroup, we mean a subsemigroup $M$ satisfying the condition that $x+y,y \in M$ implies $x \in M$.} containing $x$ instead.
\end{mydef}

\begin{pro}\label{radicaldescription}
Let $A$ be a semiring and $I$ be a $k$-ideal of $A$.  Let $J$ be a finitely generated $k$-ideal.  Then, $J\subseteq \sqrt{I}$ if and only if there is some $n>0$ such that $J^n\subseteq I$.
\end{pro}
\begin{proof}
Suppose $J^n\subseteq I$ and let $\mathfrak{p}$ be a prime $k$-ideal such that $I\subseteq \mathfrak{p}$.  Then $J^n\subseteq \mathfrak{p}$, and hence  $J\subseteq \mathfrak{p}$. Therefore, we have that $J\subseteq \sqrt{I}$.
	
Before turning to the converse, define a $k$-ideal $I$ to be $J$-less if $J^n\not\subseteq I$ for all $n>0$. We claim a filtered union of $J$-less $k$-ideals is $J$-less.  Let $\Gamma$ be a directed set and $\{N_i\mid i\in \Gamma\}$ a filtered family of $J$-less $k$-ideals.  Suppose $J^n\subseteq \bigcup_{i\in\Gamma} N_i = \sum_{i\in\Gamma} N_i$.  Then since $J^n$ is finitely generated, it is contained in some finite subsum\footnote{This is essentially the compactness condition of Definition \ref{compactdef}.}.  Since the family of $N_i$ is filtered, $J^n\subseteq N_i$ for some $i$.  This contradiction establishes the claim.  The claim together with Zorn's lemma establishes that every $J$-less $k$-ideal is contained in a maximal $J$-less $k$-ideal.
	
For the converse, suppose $J\subseteq \sqrt{I}$ and suppose for the sake of contradiction that $I$ is $J$-less.  Let $\mathfrak{p}$ be a maximal $J$-less $k$-ideal containing $I$.  To see $\mathfrak{p}$ is a prime $k$-ideal, let $\mathfrak{a},\mathfrak{b}$ be $k$-ideals with $\mathfrak{a}\mathfrak{b}\subseteq \mathfrak{p}$ and suppose $\mathfrak{a},\mathfrak{b}\not\subseteq \mathfrak{p}$.  Since $\mathfrak{p}\subsetneq \mathfrak{p}+\mathfrak{a}$ it follows from maximality that $\mathfrak{p}+\mathfrak{a}$ is not $J$-less, so there is some $k>0$ such that $J^k\subseteq \mathfrak{p}+\mathfrak{a}$.  Similarly, there is $l>0$ such that $J^l\subseteq \mathfrak{p}+\mathfrak{b}$.  A simple computation shows that
\[
J^{k+l}\subseteq (\mathfrak{p}+\mathfrak{a})(\mathfrak{p}+\mathfrak{b})\subseteq \mathfrak{p}.
\]
This contradicts the $J$-lessness of $\mathfrak{p}$, so our assumption that $\mathfrak{a},\mathfrak{b}\not\subseteq \mathfrak{p}$ is false and $\mathfrak{p}$ is a prime $k$-ideal containing $I$.  Then $J\subseteq \sqrt{I}\subseteq \mathfrak{p}$, contradicting the $J$-lessness of $\mathfrak{p}$ again, so the assumption that $I$ is $J$-less is false.  This establishes the result.
\end{proof}

\subsubsection{Radical idealic semirings, bounded distributive lattices and spectral spaces}

\begin{mydef}\cite{takagi2010construction} \label{definition: idealic}
Let $A$ be an idempotent semiring. $A$ is said to be \emph{idealic} if $1$ is the maximum element, i.e., $x\leq 1$ for all $x\in A$.
\end{mydef}

\begin{mydef}
An idealic semiring is said to be \emph{radical} if $x^2=x$ for all $x\in A$.
\end{mydef}

Some simple examples of idealic semirings are the semiring of all ideals, the semiring of finitely generated ideals and the semiring of finitely generated $k$-ideals of a semiring.  We will later see that radicals of finitely generated $k$-ideals form a radical idealic semiring.

The next result implies that like addition, the multiplication operation on a radical idealic semiring is determined by the partial order.

\begin{lem}
Let $A$ be a radical idealic semiring and $x,y\in A$.  Then $xy$ is the greatest lower bound of $x$ and $y$.
\end{lem}
\begin{proof}
$xy \leq x$ follows from $y \leq 1$, and similarly $xy \leq y$. Hence $xy$ is a lower bound.  Let $z\in A$ satisfy $z\leq x$ and $z\leq y$.  Then $z = z^2 \leq xy$, showing that $xy$ is the greatest lower bound. 
\end{proof}

The following theorem says that essentially radical idealic semirings are the same as bounded distributive lattices. 

\begin{mythm}\label{theorem: equivalence of cat}
There is an equivalence of categories between radical idealic semirings (as a subcategory of the category of semirings) and bounded distributive lattices. Furthermore, this equivalence commutes with forgetful functors (i.e. it is the identity on the level of sets). In particular, as the category of bounded distributive lattices is antiequivalent to the category of spectral spaces, we can conclude that the category of radical idealic semirings is also antiequivalent to the category of spectral spaces.
\end{mythm}
\begin{proof}
	
Let $\mathcal{S}$ be the category of radical idealic semirings and $\mathcal{L}$ be the category of bounded distributive lattices. Define a functor $F$ from $\mathcal{S}$ to $\mathcal{L}$ by $F(A)=A$ for objects and $F(f)=f$ for morphisms. We first show that our functor $F$ is well defined. 

Let $A$ be a radical idealic semiring. Then $A$ has greatest lower bounds and least upper bounds (the product and sum resp.) of all two element sets, and has a  minimum and maximum (0 and 1).  Hence $A$ is a bounded lattice. $A$ is distributive because multiplication distributes over addition. Morphisms of radical idealic semirings are maps which preserve addition and multiplication, (i.e. joins and meets), and preserve 0 and 1 (i.e. the minimum and maximum) i.e., they are the same as morphisms of bounded lattices. 

It is clear from the definition that $F$ is fully faithful. For essential surjectivity, let $L$ be a bounded distributive lattice.  Let $A=L$ be an idempotent semiring with addition and multiplication defined as  the join and meet operations.  Since $L$ is a join semilattice, $A$ is a commutative idempotent semigroup.  The multiplication (i.e. meet) is associative and has as identity the  maximal element.  Distributivity of $L$ gives the distributive law in $A$.  $A$ is idealic because $1$ is the maximal element, and is radical because meet is an idempotent operator. Finally, it is clear that the functor $F$ commutes with forgetful functors. 
\end{proof}
 
We now proceed to show explicitly how every spectral space arises from a radical idealic semiring. 
 
\subsubsection{Algebraic Lattices}

We recall some basic results from lattice theory. Most of the results in this subsection are standard results of lattice theory. We only include them here for completeness. We refer the reader to \cite{steinberg2010lattice} for more details.

\begin{mydef}\label{compactdef}
Let $L$ be a complete lattice.  Let $x \in L$.  
\begin{enumerate}
\item 
A \emph{cover} of $x$ is a family of elements $y_i \in L$ indexed by some set $\Gamma$ such that  $x\leq \displaystyle\vee_{i\in \Gamma} y_i$.	
\item 
$x$ is \emph{compact} if every cover has a finite subcover, i.e., there is some finite $I\subseteq \Gamma$ such that  $\{y_i\mid i\in I\}$ is a cover of $x$.
\end{enumerate} 
\end{mydef}

Typically compactness agrees with some more concrete notion of finite generation, as shown in the following example.

\begin{myeg}\label{idealcompact}
Let $A$ be a semiring and $I'(A)$ be the semiring of ideals of $A$, where addition is given by the sum of two ideals and multiplication is given by the product of two ideals. Then, an element $x\in I'(A)$ is compact if and only if it is finitely generated. In fact, let $I$ be a finitely generated ideal of $A$ and let $\{J_i \mid i \in \Gamma\}$ be a cover, in other words, we assume that 
\[
I\subseteq\sum_{i\in\Gamma} J_i
\]
Any element $x$ of $I$ is a finite linear combination of elements of the $J_i$, which necessarily involves only finitely many of the $J_i$.  Thus each generator of $I$ is covered by a finite subfamily of $\{J_i\mid i \in \Gamma\}$. Taking the union over all generators of these gives us a finite subcover of $I$.
	
Conversely let $I$ be a compact element of $I'(A)$. Let $\Gamma=2^I$.  For any $i\in\Gamma$, let $J_i$ be the ideal generated by $i\subseteq I$, where $i$ is a finite subset of $I$. Clearly $I\subseteq \sum_{i\in\Gamma} J_i$, since for any $x\in I$, $x\in J_{\{x\}}$.  By compactness, there is a finite subset $\Gamma'$ such that $I\subseteq \sum_{i\in\Gamma'} J_i$.  On the other hand, by construction, $J_i\subseteq I$ for all $i$, so $I=\sum_{i\in\Gamma'} J_i$ is a finite sum of finitely generated ideals, in particular, $I$ is finitely generated. 
	
In addition, the same works for the semiring $I(A)$ of $k$-ideals - an element is compact if it is the smallest $k$-ideal containing some finite set of generators.  The only difference in this proof is that if $I\subseteq\sum_{i\in\Gamma} J_i$ and $x\in I$, then there is some $z$ such that both $z$ and $x+z$ are finite linear combinations of elements of the $J_i$\footnote{The proof is omitted for brevity, but involves using Proposition \ref{saturation}}.
\end{myeg}

\begin{mydef}
A lattice is \emph{algebraic} if it is complete and every element is a least upper bound of compact elements.
\end{mydef}

\begin{myeg}
The semiring of ideals of a semiring $A$ is a complete
lattice with the following operations
\begin{equation}\label{jm}
	\vee_{i\in \Gamma}I_i := \sum_{i\in \Gamma}I_i \qquad\qquad \wedge_{i\in \Gamma}I_i := \bigcap_{i\in \Gamma}I_i 
\end{equation}
Moreover, this complete lattice is also algebraic
for every ideal is a sum of finitely generated (and hence compact) ideals, for instance the sum of all principal subideals.
\\ The semiring of all $k$-ideals of a semiring $A$ is also an algebraic lattice with similar operations.
\end{myeg}

\begin{mydef}
For an algebraic lattice $L$, we will let $L^c$ be the set of compact elements.  For a commutative idempotent semigroup $M$, we will let $S(M)$ be the complete lattice (with operations similar to \eqref{jm}) of $k$-subsemigroups of $M$.
\end{mydef}

Note that the same proof as for ideals (as in Example \ref{idealcompact}) shows that $S(M)^c$ is the set of finitely generated $k$-subsemigroups of $M$, and that $S(M)$ is algebraic.  Similarly, the following lemma shows $L^c$ is an idempotent semigroup.

\begin{lem}\label{als}
Let $L$ be an algebraic lattice.  Then $L^c$ is a commutative idempotent semigroup under the order induced by $L$.
\end{lem}
\begin{proof}
As a lattice, $L$ is an idempotent semigroup, so we only need to show $L^c$ is closed under addition and contains $0$.  Of course, $0$ is compact because every cover contains an empty subcover.  If $x,y\in L^c$, then every cover of $x+y$ covers $x$ and $y$.  Thus we may pick two finite subcover which cover $x$ and $y$ respectively.  Their union covers $x+y$.
\end{proof}

The following well-known result shows that the study of algebraic lattices is equivalent to the study of commutative idempotent semigroups. We include a proof just for completeness.

\begin{mythm}\label{alglat} $ $
\begin{enumerate}
\item
Let $M$ be a commutative idempotent semigroup.  Then $M\cong S(M)^c$ as semigroups.
\item 
Let $L$ be an algebraic lattice.  Then there is a lattice isomorphism $L\cong S(L^c)$.
\end{enumerate}
\end{mythm}
\begin{proof}
$(1)$ Define $f: M \rightarrow S(M)^c$ by letting $f(x)$ be the $k$-subsemigroup generated by $x$.  To see $f$ is a homomorphism, note that $x+y\in f(x)+f(y)$, so $f(x+y)\subseteq f(x)+f(y)$.  Conversely $x$ is in any $k$-subsemigroup containing a larger element such as $x+y$, so $f(x),f(y)\subseteq f(x+y)$, showing that $f(x)+f(y) \subseteq f(x+y)$.  In addition, $f(0) = 0$.
	
The image of $f$ is the set of principal $k$-subsemigroups. Since $f$ is a homomorphism, this is closed under addition, so every finitely generated $k$-subsemigroup is in the image of $f$ (and is principal). Thus $f$ is surjective. For injectivity, note that for any $x\in M$, the set $M_x = \{y\in M\mid y\leq x\}$ is a $k$-subsemigroup with maximal element $x$ (in fact $M_x=f(x)$).  If $f(x)=f(y)$, any $k$-subsemigroup containing $x$ contains $y$. So, $y\in M_x$ and thus $y\leq x$.  Similarly $x\leq y$ so $x=y$, showing that $f$ is injective as well. 
	
$(2)$ Define the following function:
\[
f: L \rightarrow S(L^c), \quad x \mapsto \{y\in L^c \mid y\leq x\}.
\]
We first claim that $f$ is well defined. By the definition, $f(x)$ is closed under finite joins (i.e. under addition).  Clearly $f(x)$ contains the minimum element $0\in L^c$.  If $y\in f(x)$ and $z\in L^c$ satisfies $z\leq y$, then $z\in f(x)$.  Hence $f(x)$ is a $k$-subsemigroup. Note that any $x\in L$ is a join of compact elements, so satisfies $x=\vee_{y\in f(x)} y$.  Hence $f$ is injective.  For surjectivity, let $M\subseteq S(L^c)$ be a $k$-subsemigroup.  Let $x = \vee_{y\in M} y$.  Clearly $M\subseteq f(x)$.  Conversely if $y\in f(x)$, then $y\leq x$, and $M$ covers $y$. Since $y$ is compact, there is a finite subcover, i.e., $y$ is bounded by a finite sum of elements of $M$.  Since $M$ is a  $k$-subsemigroup, this implies $y\in M$.
	
Finally, if $x\leq y$, it is clear that $f(x)\leq f(y)$.  Conversely, if $f(x)\leq f(y)$, then
\[
x = \vee_{z\in f(x)} z \leq \vee_{z\in f(y)} z = y,
\]
showing that $f$ is a lattice isomorphism. 
\end{proof}

\begin{mydef}
By an \emph{algebraic lattice with multiplication}, we mean an algebraic lattice together with a multiplication operation which has a compact identity, is associative and commutative, distributes over arbitrary joins, and preserves compactness.
\end{mydef}

It easily follows from Lemma \ref{als} that an algebraic lattice with multiplication $L$ is an idempotent semiring with some additional properties.  In particular if $x,y,z,w\in L$ with $x\leq z$ and $y\leq w$ then $xy\leq zw$.

\begin{mydef}
Let $A$ be an idempotent semiring and $M,N\in S(A)$ be additive $k$-subsemigroups. Let $MN$ be the smallest $k$-subsemigroup containing $\{xy\mid x\in M,y\in N\}$.
\end{mydef}

Theorem \ref{alglat} extends to algebraic lattices with multiplication and idempotent semirings.  First we will need a lemma to understand the multiplication in $S(A)$.

\begin{lem}\label{satmult} $ $
\begin{enumerate}
 \item Let $A$ be an idempotent semiring and $M,N\in S(A)$.  If $S$ and $T$ are generating sets of $M$ and $N$ respectively, then $MN$ is the smallest $k$-subsemigroup containing $\{xy\mid x\in S,y\in T\}$.
 \item Let $A$ be an idempotent semiring and $M,N\in I(A)$.  If $S$ and $T$ are generating sets of $M$ and $N$ respectively, then $MN$ is the smallest $k$-ideal\footnote{Note that the product $MN$ here is the product of $k$-ideals, which differs from the first part of the lemma}  containing $\{xy\mid x\in S,y\in T\}$.
\end{enumerate}
\end{lem}
\begin{proof}
We start by proving (1).  Clearly $\{xy\mid x\in S,y\in T\}\subseteq MN$.  Before proving the reverse inclusion, for each $x\in A$ and $C\in S(A)$, we define the following subset of $A$: 
\[
(C:x)=\{y\in A\mid xy\in C\}.
\]
Since multiplication by $x$ is a semigroup homomorphism, $(C:x)$ is the preimage of a $k$-subsemigroup under a homomorphism, and hence $(C:x)$ is a $k$-subsemigroup.
	
Let $C$ be a $k$-subsemigroup containing $\{xy\mid x\in S,y\in T\}$.  Then, we have that
\[
S\subseteq \bigcap_{y\in T} (C:y).
\] 
Since the right side is a $k$-subsemigroup, $M\subseteq \bigcap_{y\in T} (C:y)$.  Thus $\{xy\mid x\in M,y\in T\}\subseteq C$.  Applying the same trick a second time shows $\{xy\mid x\in M,y\in N\}\subseteq C$.

If $C$ is a $k$-ideal, then the above shows $(C:x)$ is a $k$-subsemigroup and a routine calculation shows $(C:x)$ is closed under multiplication, so it is a $k$-ideal.  The rest of the proof of (2) is similar to the first part.
\end{proof}

We now give two examples of algebraic lattices with multiplication.

\begin{lem} \label{lemma: S(A),I(A) algebraic lattices}
Let $A$ be an idempotent semiring. Then, $S(A)$ and $I(A)$ are algebraic lattices with multiplication.
\end{lem}
\begin{proof}
The case of $I(A)$ is proven similarly to $S(A)$, so we will only show $S(A)$ is an algebraic lattice with multiplication.  In fact, it follows from Lemma \ref{satmult} that the product of compact elements is compact.  The $k$-subsemigroup generated by $1$ is the identity and is compact. Additionally, commutativity is clear.  For associativity, let $L,M,N\in S(A)$.  Define 
\[
g(L,M) = \{xy\mid x\in L,y\in M\}.
\]
$g(L,M)$ generates $LM$.  Then by Lemma \ref{satmult}, we have that $(LM)N$ is generated by
\[
\{xc\mid x\in g(L,M), c\in N\}=\{abc\mid a\in L, b\in M, c\in N\}.
\]
By symmetry, this generates $L(MN)$ as well, proving the associativity. 
	
It remains to show distributivity.  Let $M\in S(A)$ and $\{N_i\mid i\in \Gamma\}\subseteq S(A)$.  Let $N = \sum_i N_i$.  Then by Lemma \ref{satmult}, $MN$ is the smallest $k$-subsemigroup containing the following set:
\[
\{xy\mid x\in M,y\in \bigcup_i N_i\}=\bigcup_i \{xy\mid x\in M,y\in N_i\}.
\]
But this is the sum of the $k$-subsemigroups generated by $\{xy\mid x\in M,y\in N_i\}$. In particular, we have that
\[
M\left(\sum_i N_i \right) = MN = \sum_i \left(MN_i\right).
\]
\end{proof}

\begin{mythm}\label{theorem: allgat2}
Let $A$ be an idempotent semiring and $L$ be an algebraic lattice with multiplication. 
\begin{enumerate}
\item
$S(A)$ is an algebraic lattice with multiplication, and $A\cong S(A)^c$ as semirings.
\item
$L^c$ is an idempotent semiring, and there is a lattice isomorphism $L\cong S(L^c)$ which preserves multiplication.  
\end{enumerate}
\end{mythm}
\begin{proof}
$(1)$ From Lemma \ref{lemma: S(A),I(A) algebraic lattices}, we already know that $S(A)$ is an algebraic lattice with multiplication. Furthermore, we know that there is an isomorphism of semigroups from $A$ to $S(A)^c$ sending $x$ to $\langle x \rangle\in S(A)^c$, the smallest $k$-subsemigroup containing $x$.  It follows from Lemma \ref{satmult} that $\langle xy \rangle = \langle x \rangle \langle y \rangle$, proving that this isomorphism is indeed an isomorphism of semirings. 

$(2)$ It is clear that $L^c$ is an idempotent semiring. For the second assertion, we define the following function:
\[
f: L \rightarrow S(L^c), \quad x \mapsto \{y\in L^c \mid y\leq x\}.
\]
From Lemma \ref{satmult}, $f$ is an isomorphism of lattices. We only have to prove that $f$ is compatible with multiplication. In fact, if $a\in f(x)$ and $b\in f(y)$, then $ab$ is compact and $ab\leq xy$. So, $ab\in f(xy)$.  Thus, $f(x)f(y)\subseteq f(xy)$. 

For the reverse inclusion, let $c\in f(xy)$.  We may write $y$ as a join of compact elements $y=\vee_i z_i$.  Then the elements $xz_i$ cover $c$ and so it is covered by finitely many.  Setting $z$ to be the join of these finitely many $z_i$, we get $c\leq xz$ with $z$ compact and $z\leq y$.  Write $x=\vee_i w_i$ with $w_i$ compact.  As before, picking a finite subcover of $w_i z$ gives some compact $w$ with $w\leq x$ and $c\leq wz$.  We have $w\in f(x)$ and $z\in f(y)$ and so $c\in f(x)f(y)$, showing that $f(xy) \subseteq f(x)f(y)$. 
\end{proof}

\begin{rmk}\label{satimplideal}
When applying the above result to ideals, it is sometimes worth noting that any $k$-subsemigroup $I$ of an idealic semiring $A$ is an ideal.  This is because for any $r\in A, x\in I$, $rx\leq x$ and so $rx\in I$.
\end{rmk}

\begin{myeg}\label{idealcorr}
Let $I(A)$ be the lattice of $k$-ideals of a semiring $A$.  Then $$I(A)\cong S(I(A)^c) = I(I(A)^c)$$ by Remark \ref{satimplideal} and Theorem \ref{alglat}.
\end{myeg}

\subsubsection{Idealization, Radicalization, and Zariski space}

In this subsection, we prove that for a given spectral space $X$, there exists an idempotent semiring $A$ whose prime $k$-spectrum is homeomorphic to $X$. To this end, we first introduce two operations, \emph{idealization} and \emph{radicalization}.

\begin{mydef}
Let $A$ be an idempotent semiring.  The \emph{idealization} of $A$ is the quotient of $A$ by the congruence generated by the relations of the form $x+1\sim 1$ for all $x$ in $A$. 
\end{mydef}

It may be easily verified that the idealization of an idempotent semiring $A$ is the initial object in the category of idealic $A$-algebras\footnote{By an $A$-algebra $B$, we mean a semiring $B$ with a unit map $f:A \to B$.}% Proposition \ref{pro: idealization} shows that the idealization always exists.}. Note that if $A$ is an idealic semiring, then clearly $A$ satisfies the universal property above.
It turns out that the idealization of an idempotent semiring $A$ is the same as the semiring of finitely generated ideals of $A$ as the following proposition shows.

\begin{pro}\label{pro: idealization}
	Let $A$ be an idempotent semiring and $I(A)$ the semiring of $k$-ideals.  Then $I(A)^c$ is the idealization of $A$.
\end{pro}
\begin{proof}
	Clearly $I(A)^c$ is idealic and is equipped with an $A$-algebra structure by the map sending an element $x\in A$ to the corresponding principal $k$-ideal $\langle x \rangle$ as in Theorem \ref{alglat}. 
	
	Let $S$ be an idealic $A$-algebra, and let $f:A\rightarrow S$ be the unit of its algebra structure.  Define the function $f_*: I(A)^c\rightarrow I(S)^c$ as follows.  Let $I\in I(A)^c$, and let $X$ be a generating set of $I$.  Define $f_*(I)$ to be the $k$-ideal generated by $\{f(x)\mid x\in X\}$.  To show this is well defined, observe that $f^{-1}(f_*(I))$ is a $k$-ideal (since its the preimage of one) and contains $X$.  Hence $I\subseteq f^{-1}(f_*(I))$ and so $\{f(x)\mid x\in I\}\subseteq f_*(I)$.  Since $\{f(x)\mid x\in I\}$ contains a generating set, $f_*(I)$ may be described in a choice-free way as the $k$-ideal generated by $\{f(x)\mid x\in I\}$.  Since we can choose a finite generating set $X$, $f_*(I)$ is finitely generated.  
	
	Let $I, J\in I(A)^c$ and fix generating sets $X\subseteq I$ and $Y\subseteq J$. Then, $X \cup Y$ generates $I+J$ and so $f_*(I+J)$ is generated by $\{f(x)\mid x\in X\} \cup \{f(y)\mid y\in Y\}$.  Hence $f_*(I+J)=f_*(I)+f_*(J)$.  Similarly $IJ$ is generated by $XY$, so $f_*$ preserves multiplication as well.  It clearly preserves $0$ and $1$, and hence $f_*$ is a homomorphism.  Furthermore taking $X=\{x\}$ we see $f_*(\angles{x})$ is the principal $k$-ideal generated by $f(x)$.  
	
	Combining the above with Example \ref{idealcorr} gives homomorphisms $I(A)^c\rightarrow I(S)^c \cong S$ whose composition sends $\langle x \rangle$ to $f(x)$ and hence is an $A$-algebra homomorphism.  The uniqueness part of the universal property follows from the fact that there is only one homomorphism sending $\langle x \rangle$ to $f(x)$ since the principal ideals generate $I(A)^c$ (and in fact are all the elements).
\end{proof}

Next, we introduce a key definition in proving our main theorem in this section. 

\begin{mydef}
Let $A$ be a semiring.  $I_{rad}(A)$ is defined to be the lattice of radical ideals\footnote{The definition of radical ideals assumes that they are $k$-ideals. To be specific, by a radical ideal we mean that a $k$-ideal $I$ such that $\sqrt{I}=I$. Equivalently, $I=\sqrt{J}$ for some $k$-ideal $J$.} together with the multiplication operation sending a pair $(I,J)$ to $\sqrt{IJ}$.
\end{mydef}

The meet and join operations in $I_{rad}(A)$ are $(I,J) \mapsto I\cap J$ and $(I,J)\mapsto \sqrt{I+J}$.  It is easy to show that these descriptions apply to infinite meets and joins as well, so $I_{rad}(A)$ is complete.

\begin{lem}\label{lemma: 3.34}
Let $A$ be a semiring. 
\begin{enumerate}
	\item
	The compact elements of $I_{rad}(A)$ are the radicals of finitely generated $k$-ideals.
	\item
	$I_{rad}(A)$ is an algebraic lattice.
	\item
 $I_{rad}(A)$ is an algebraic lattice with multiplication (and hence an idempotent semiring).
\end{enumerate}
\end{lem}
\begin{proof}
$(1)$ and $(2)$: Any radical ideal $I$ can be written as $I=\sqrt{\sum_{x\in I} \langle x \rangle}$ and so it is covered by radicals of finitely generated $k$-ideals.  Thus (1) implies (2).  This remark also shows that any compact element is the radical of a finitely generated $k$-ideal.

Conversely, let $I\in I_{rad}(A)$ be the radical of a finitely generated $k$-ideal $J$.  Let $N_i$ for $i \in\Gamma$ be a cover of $I$ in $I_{rad}(A)$.  Then we have 
\[
J \subseteq \sqrt{J} = I\subseteq \sqrt{\sum_{i\in\Gamma} N_i}.
\]
By Proposition \ref{radicaldescription}, there is some $n > 0$ such that $J^n\subseteq\sum_{i\in\Gamma} N_i$.  Since $J^n$ is finitely generated, $J^n\subseteq\sum_{i\in\Gamma'} N_i$ for some finite $\Gamma'\subseteq \Gamma$.  Taking radicals, $I\subseteq \sqrt{\sum_{i\in\Gamma'} N_i}$.  Thus $I$ is compact.
	
(3): To show multiplication preserves compactness, let $\mathfrak{a},\mathfrak{b}\in I_{rad}(A)^c$.  Then we can choose finitely generated ideals $I,J$ such that $\mathfrak{a}=\sqrt{I}$ and $\mathfrak{b}=\sqrt{J}$.  Then the product $\mathfrak{a}\mathfrak{b}$ is defined as the ideal $\sqrt{\sqrt{I}\sqrt{J}}$.  It is easy to see that this equals $\sqrt{IJ}$ - since both ideals are radical, we only need to check they belong to the same primes.  But since $IJ$ is finitely generated $\mathfrak{a}\mathfrak{b}$ is compact.  Since the identity is $A = \sqrt{A}$, it is compact as well.

Associativity amounts to the statement that
\[
\sqrt{\sqrt{IJ}N}=\sqrt{I\sqrt{JN}}.
\]
This can be checked by noting that a prime $k$-ideal $\mathfrak{p}$ contains $\sqrt{IJ}N$ if and only if $\sqrt{IJ} \subseteq \mathfrak{p}$ or $N\subseteq \mathfrak{p}$ if and only if $IJ\subseteq \mathfrak{p}$ or $N\subseteq \mathfrak{p}$ if and only if at least one of $I$, $J$, $N$ is a subset of $\mathfrak{p}$. By a similar argument, this holds if and only if $I\sqrt{JN}\subseteq \mathfrak{p}$.  Thus $\sqrt{IJ}N$ and $I\sqrt{JN}$ are contained in the same prime $k$-ideals, so have the same radical.  All that remains is to prove distributivity.

Let $I$ be a radical ideal and $\{J_i\mid i\in \Gamma\}$ be a collection of radical $k$-ideals.  After unpacking the distributive law, what we must show is that
\[
\sqrt{\sum_i \sqrt{IJ_i}}=\sqrt{I\sqrt{\sum J_i}}.
\]
Since the radical of a $k$-ideal is the intersection of the prime $k$-ideals containing it, we must show $\sum_i \sqrt{IJ_i}$ and $I\sqrt{\sum J_i}$ are contained in the same  prime $k$-ideals.  

Let $\mathfrak{p}$ be a prime $k$-ideal such that $\sum_i \sqrt{IJ_i}\subseteq \mathfrak{p}$. Then, $IJ_i\subseteq \mathfrak{p}$ for all $i$ since $\mathfrak{p}$ is a $k$-ideal.  If $I\not\subseteq p$, then $J_i\subseteq \mathfrak{p}$ for all $i$ and hence $\sum_i J_i\subseteq \mathfrak{p}$.  Hence $\sqrt{\sum_i J_i}\subseteq \mathfrak{p}$ and so $I\sqrt{\sum J_i} \subseteq \mathfrak{p}$.  The case where $I\subseteq \mathfrak{p}$ is trivial.

Now, suppose that $\mathfrak{p}$ is a prime $k$-ideal such that $I\sqrt{\sum J_i} \subseteq \mathfrak{p}$.  If $I\not\subseteq \mathfrak{p}$, then $\sqrt{\sum J_i} \subseteq \mathfrak{p}$ and hence $\sum J_i$ are contained in $\mathfrak{p}$.  Thus in this case, $IJ_i\subseteq \mathfrak{p}$ for all $i$, and establishing this in the case $I\subseteq \mathfrak{p}$ is easier.  It follows that $\sqrt{IJ_i}$ and hence $\sum_i \sqrt{IJ_i}$ are contained in $\mathfrak{p}$.
\end{proof}

Now, we introduce the second operation. 

\begin{mydef}
Let $A$ be an idealic semiring.  The \emph{radicalization} of $A$ is the quotient of $A$ by the congruence generated by relations of the form $x^2\sim x$ for all $x$ in $A$.
\end{mydef}

It may be easily verified that the radicalization of $A$ is the initial object in the category of radical idealic $A$-algebras.  Note that if $A$ is a radical idealic semiring, then clearly $A$ satisfies the universal property above.

When $A$ is an idempotent semiring and $\sim$ is a congruence relation on $A$, the quotient semiring $A/\sim$ is also idempotent, and hence equipped with a partial order: $[a] \leq [b]$ if $[a]+[b]=[b]$, where $[a]$ is the equivalence class of $a$. In particular, if $a \leq b$, then $[a] \leq [b]$. 

\begin{pro}
Let $A$ be an idealic semiring.  Then the radicalization of $A$ is isomorphic to $I_{rad}(A)^c$.
\end{pro}
\begin{proof}
Using the isomorphism $A\cong I(A)^c$, we may view the radicalization as the quotient of $I(A)^c$ by the congruence $\equiv$ generated by the relations $x^2\equiv x$.  Define the equivalence relation $\sim$ on $I(A)^c$ as follows:
\[
I\sim J \iff \sqrt{I}=\sqrt{J}.
\] 
The equation $\sqrt{I^2}=\sqrt{I}$ shows that $I\equiv J$ implies $I\sim J$. 

Conversely, suppose that $I \sim J$ for $I,J \in I(A)^c$.  Since $\sqrt{I}=\sqrt{J}$, we have that $I\subseteq \sqrt{J}$.  By Proposition \ref{radicaldescription} and Lemma \ref{lemma: 3.34}, there is some $n>0$ such that $I^n\subseteq J$.  This implies $[I]\leq [J]$, where the brackets denote the equivalence classes inside the quotient semiring $I(A)^c/\equiv$.  Similarly, we have $[J] \leq [I]$ and so $I\equiv J$.  Thus the two equivalence relations are equal, and the radicalization is $I(A)^c/\sim$.  It is easy to check that this is isomorphic to $I_{rad}(A)^c$.
\end{proof}

\begin{pro}\label{radquot}
Let $A$ be an idempotent semiring.  Then $I_{rad}(A)^c$ is isomorphic to the quotient of $A$ by the congruence generated by relations of the form $x^2\sim x$ and $x+1\sim 1$ for all $x$ in $A$.
\end{pro}
\begin{proof}
Instead of quotienting by all relations at once, we may first quotient by relations of the form $x+1\sim 1$ then by relations of the form $x^2\sim x$.  In other words, the quotient semiring is the radicalization of the idealization of $A$.  This is $I_{rad}(I(A)^c)^c$, and we must show this semiring is isomorphic to $I_{rad}(A)^c$.  
	
We know there is an isomorphism $I(I(A)^c)\cong I(A)$.  It clearly preserves primality, and the description of radical ideals as intersections of prime ideals implies it must preserve radicalness as well.  Thus we have $I_{rad}(I(A)^c)\cong I_{rad}(A)$.  As an isomorphism, this preserves infinite joins, so preserves compactness and we have $I_{rad}(I(A)^c)^c\cong I_{rad}(A)^c$
\end{proof}

%\subsection{The Zariski Space}

For an idempotent semiring $A$, we call $A$ complete if the least upper bound exists for any subset $M$ of $A$ with respect to the canonical partial order of $A$. Now, we introduce the Zariski space of a complete idealic semiring.  Such spaces have been studied in \cite{takagi2010construction}.

\begin{mydef}
Let $A$ be a complete idealic semiring. 
\begin{enumerate}
\item 
An element $\mathfrak{p}\in A$ is called \emph{prime} if $xy\leq \mathfrak{p}$ implies $x\leq \mathfrak{p}$ or $y\leq \mathfrak{p}$.	
\item 
The \emph{Zariski space} of $A$, denoted $\mathrm{Zar}(A)$, is the set of prime elements of $A$ with the topology given by closed subsets of the form $V(x) = \{\mathfrak{p} \textrm{ prime}\mid x\leq \mathfrak{p}\}$ for $x\in A$.	
\end{enumerate}
\end{mydef}

It is easy to check that this does define a topology.  The completeness is needed for closure under infinite intersections, while being idealic ensures that the empty set is closed.

For a semiring $A$, it is clear that the idempotent semirings $I(A)$ and  $I_{rad}(A)$ are both complete and idealic. The prime $k$-spectrum $\mathrm{Spec}_k A$ of a semiring $A$ is homeomorphic to the Zariski space of $I(A)$ as we will show next. Moreover, we will also show that $\mathrm{Zar}(I_{rad}(A))$ is homeomorphic to $\mathrm{Zar}(I(A))$ (hence homeomorphic to $\mathrm{Spec}_k A$). This is analogous to the fact that $\Spec R$ is homeomorphic to $\Spec R_{\textrm{rad}}$ for a commutative ring $R$.

\begin{pro}\label{proposition: zar}
Let $A$ be a semiring. Then, the prime $k$-spectrum $\mathrm{Spec}_k A$ is homeomorphic to the following Zariski spaces:
$$ \mathrm{Spec}_k A\cong \mathrm{Zar}(I(A)) \cong \mathrm{Zar}(I_{rad}(A))$$
\end{pro}
\begin{proof}
By definition, any prime element $\mathfrak{p}$ of $\mathrm{Zar}(I(A))$ is a prime $k$-ideal $\mathfrak{p}$ of $A$ and vice versa. Also, for any element $I\in I(A)$ (or equivalently, any $k$-ideal $I$ of $A$), clearly, $I\leq\mathfrak{p}$ in $I(A)$ is same as $I\subseteq \mathfrak{p}$ as $k$-ideals. Thus, $ \mathrm{Spec}_k A$ is homeomorphic to $\mathrm{Zar}(I(A))$.\\
Now, if $\mathfrak{p}$ is a prime element of $I(A)$, then $\mathfrak{p}$ is radical and so $\mathfrak{p}\in I_{rad}(A)$.  If $I,J\in I_{rad}(A)$ satisfy $IJ\leq \mathfrak{p}$ in $I_{rad}(A)$, then $\sqrt{IJ}\subseteq \mathfrak{p}$ and so $I\subseteq \mathfrak{p}$ or $J\subseteq \mathfrak{p}$.  Hence $\mathfrak{p}$ is prime in $I_{rad}(A)$.  
	
Conversely, let $\mathfrak{p}$ be prime in $I_{rad}(A)$. If $I,J\in I(A)$ satisfy $IJ\subseteq \mathfrak{p}$, then $\sqrt{\sqrt{I}\sqrt{J}}=\sqrt{IJ}\subseteq \mathfrak{p}$ because $IJ$ and $\sqrt{I}\sqrt{J}$ are contained in the same prime $k$-ideals and because $\mathfrak{p}$ is radical.  Then either $\sqrt{I}\subseteq \mathfrak{p}$ or $\sqrt{J}\subseteq \mathfrak{p}$.  This implies $I\subseteq \mathfrak{p}$ or $J\subseteq \mathfrak{p}$.  Thus $ \mathrm{Zar}(I(A)) = \mathrm{Zar}(I_{rad}(A))$ as sets.
	
An ideal $I$ induces the same closed subset of $ \mathrm{Zar}(I(A))$ as its radical and so every closed subset has the form $V(I)=\{\mathfrak{p}\,\mathrm{prime}\mid I\subseteq \mathfrak{p}\}$ for some radical ideal $I$.  Similarly, every closed subset of $\mathrm{Zar}(I_{rad}(A))$ has the form $\{\mathfrak{p}\,\mathrm{prime}\mid \sqrt{I}\subseteq \sqrt{\mathfrak{p}}\}$ for some radical ideal $I$.  But $\sqrt{I}\subseteq \sqrt{\mathfrak{p}}$ is the same as $I\subseteq \mathfrak{p}$ since $\mathfrak{p}$ is radical. This proves that $\mathrm{Zar}(I(A))$ is homoemorphic to $\mathrm{Zar}(I_{rad}(A))$. 
\end{proof}

\begin{mythm}
Let $A$ be an idempotent semiring.  Then the semiring $I_{rad}(A)^c$, which by Proposition \ref{radquot} is a quotient of $A$, has the same prime $k$-spectrum as $A$, i.e. there is a homeomorphism: $$\mathrm{Spec}_k A \cong \mathrm{Spec}_k \left( I_{rad}(A)^c \right).$$
\end{mythm}
\begin{proof}
By the theory of algebraic lattices, $I(I_{rad}(A)^c)\cong I_{rad}(A)$.  Thus, from Proposition \ref{proposition: zar}, we have
\[
\mathrm{Spec}_k A\cong \mathrm{Zar}(I_{rad}(A))\cong \mathrm{Zar}(I(I_{rad}(A)^c))\cong\mathrm{Spec}_k \left( I_{rad}(A)^c\right).
\]
\end{proof}

We have shown that the category of spectral spaces is antiequivalent to the category of radical idealic semirings by Theorem \ref{theorem: equivalence of cat}.  We are now finally ready to state explicitly how every spectral space arises from a radical idealic semiring. %though this is essentially a known result.\\
\vspace{2mm}

Let $X$ be a topological space.  The collection of all open sets of $X$, ordered by inclusion, is a complete lattice with {\it join} given by the union and {\it meet} given by the interior of the intersection. Let $\mathcal{O}(X)$ denote this complete lattice of open subsets of $X$.  The compact objects of $\mathcal{O}(X)$ are the compact open subsets.  We define the product of open subsets to be the intersection.

\begin{lem}Let $X$ be a spectral space.  Then $\mathcal{O}(X)$ is an algebraic lattice with multiplication.
\end{lem}
\begin{proof}
Since $X$ is spectral, then quasi-compact open subsets form a basis for $X$, so $\mathcal{O}(X)$ is algebraic.  Furthermore, the intersection of two quasi-compact open subsets of a spectral space is quasi-compact, so the product of compact elements is compact.  Additionally $X$ is the multiplicative identity, and is compact.  Distributivity and associativity follow from the set-theoretic fact that intersections are associative and distribute over unions.  
\end{proof}

$\mathcal{O}(X)$ is obtained from the lattice of closed subsets by reversing the order of inclusion.  If $X=\Spec R$ for a commutative ring $R$, there is an order reversing correspondence between closed subsets and radical ideals, so $\mathcal{O}(X)\cong I_{rad}(R)$.  Then, one has the following homeomorphisms:
\[
X=\mathrm{Spec} R \cong \mathrm{Spec}_s (I_{rad}(R)^c) \cong \mathrm{Spec}_s \left(\mathcal{O}(X)^c\right),
\]
showing that a spectral space is homeomorphic to the prime $k$-spectrum of an idempotent semiring. We now show that this isomorphism can be constructed explicitly without appeal to Hochster's theorem.

\begin{mythm}\label{theorem: spectral theorem}
Let $X$ be a spectral space and $\mathcal{O}(X)$ be the lattice of open subsets.  
\begin{enumerate}
\item
$\mathcal{O}(X)^c$ is a radical idealic semiring.
\item 
$X\cong \mathrm{Zar}(\mathcal{O}(X))\cong\mathrm{Spec}_s (\mathcal{O}(X)^c)$.	
\end{enumerate}
\end{mythm}
\begin{proof}
Part (1) is trivial.  Since $\mathcal{O}(X)$ is algebraic and $\mathcal{O}(X)^c$ is idealic as a spectral space $X$ itself is compact, from Theorem \ref{theorem: allgat2}, one has that
\begin{equation}\label{eq 1}
\mathcal{O}(X)\cong I(\mathcal{O}(X)^c),
\end{equation}
and hence
\begin{equation}\label{eq 2}
\mathrm{Zar}(\mathcal{O}(X))\cong\mathrm{Spec}_s (\mathcal{O}(X)^c).
\end{equation}
Let $\mathcal{C}(X)$ be the lattice of closed subsets with the reverse inclusion order.  Then, obviously we have that
\begin{equation}\label{eq 3}
\mathcal{C}(X)\cong \mathcal{O}(X).
\end{equation}

We claim that $X\cong \mathrm{Zar}(\mathcal{C}(X))$ as topological spaces; this will prove the desired result by \eqref{eq 2} and \eqref{eq 3}. Indeed, prime elements of $\mathcal{C}(X)$ are irreducible closed subsets of $X$, which are in one-to-one correspondence with points of $X$, so we have $X\cong \mathrm{Zar}(\mathcal{C}(X))$ as sets. To be specific, we have the following set bijection:
\[
f:X \to \mathrm{Zar}(\mathcal{C}(X)), \quad p \mapsto \bar{p},
\]
where $\bar{p}$ is the topological closure of $\{p\}$ in $X$. Then, $f$ is injective since $X$ is a $T_0$ space; $x \in \bar{y}$ and $y \in \bar{x}$ happens at the same time only when $x=y$. The function $f$ is also clearly surjective, since any irreducible closed subset $Y$ of $X$ should be of the form $\overline{\eta_{Y}}$, where $\eta_{Y}$ is the generic point of $Y$. 

Now, we show that $f$ is continuous. Suppose $Z\subseteq X$ such that $f(Z)$ is closed.  Then, we have
\[
 f(Z) = V(Y) = \{ \bar{p} \mid Y \leq  \bar{p} \} \textrm{ for some } Y \in \mathcal{C}(X). 
\]
It follows from the definition of $\leq$ in $\mathcal{C}(X)$ that
 \[
f(Z) = \{ \bar{p} \mid \bar{p} \textrm{ is a subset of $Y$} \} = \{ \bar{p} \mid p \in Y \} = f(Y),
 \]
where the second to last equality uses that $Y$ is closed.  Since $f$ is injective, we have that $Z = Y$. In particular, $Z$ is closed in $X$ and $f$ is continuous.  Showing $f^{-1}$ is continuous is similar. This proves our claim.
\end{proof}

\subsection{Spectral spaces arising from congruence relations}

In this subsection, we study spectral spaces arising from sets of congruence relations on a semiring. In particular, we prove that the set of prime congruences (as in \cite{joo2014prime}) of an idempotent semiring $A$ is a spectral space. We first recall some definitions. 

Let $A$ be an idempotent semiring. In \cite{bertram2017tropical}, Bertram and Easton introduced the notion of prime congruences on an idempotent semiring to prove a version of tropical nullstellensatz, which was then fully proved\footnote{Bertram and Easton proved one inclusion of the set equality for tropical nullstellensatz.} by Jo\'o and Mincheva \cite{joo2014prime}. One first defines the \emph{twisted product} $x \cdot_t y$ of elements $x=(x_1,x_2),y=(y_1,y_2) \in A \times A$ as follows:
\[
(x\cdot_t y):=(x_1y_1+x_2y_2,x_1y_2+x_2y_1). 
\]
Now, a congruence $C$ on $A$ is said to be \emph{prime} if $C$ is proper (i.e, $C \neq A \times A$) and satisfies the following condition:
\[
\textrm{If } x\cdot_t y \in C \textrm{ then } x \in C \textrm{ or } y \in C \quad \forall x,y \in A \times A.
\]

We let $\Spec_\mathfrak{c} A$ be the set of prime congruences. We impose the \emph{hull-kernel topology} on $\Spec_\mathfrak{c} A$, which is defined as follows:

\begin{mydef}
	Let $S$ be a set. One may impose the \emph{hull-kernel topology} on the power set $2^S$ by declaring that the open sub-basis of the topology is given by the sets of the form $$D(F):= \{I \in 2^S \mid F \not \subseteq I\},$$ where $F$ is a finite subset of $S$. We will denote by $V(F)$ the complement of $D(F)$ in $2^S$. 	
\end{mydef}

\begin{pro}\label{proposition: subcongruences}
	Let $A$ be a semiring and $C$ be a congruence on $A$. Then,
	\begin{enumerate}
		\item 
		The collection of subcongruences $S_C\subseteq 2^{(A \times A)}$ of the congruence $C$, endowed with the hull-kernel topology induced from $2^{(A \times A)},$ is a spectral space.	
		\item 
		If $C$ is finitely generated, the collection of all proper subcongruences of the congruence $C$ and the collection of all proper prime subcongruences of the congruence $C$ are spectral spaces.
	\end{enumerate}
\end{pro}
\begin{proof}
	(1) For any $a,b \in A\times A$, let $a\cdot_tb$ denote the twisted product of $a$ and $b$. The set of equivalence relations
	\begin{equation*}
	\mathcal{E}=[\underset{x,y,z\in A}\cap V(x, z) \cup D(x, y) \cup D(y, z)] \cap [\underset{x,y\in A}\cap (V(x, y) \triangle V(y, x))] \cap [\underset{x\in A}\cap V(x, x)]
	\end{equation*}
	is in the Boolean algebra generated by sets of the form $V(x, y)$ so is a patch closed subset (here $\triangle$ is the symmetric difference).  Now it can be easily seen that
	\begin{equation*}
	S_C=\mathcal{E} \cap [\underset{a,b\in A\times A}\cap D(a)\cup D(b)\cup V(a+b)] \cap [\underset{a,b\in A\times A}\cap D(a)\cup V(a\cdot_tb)]\cap [\underset{z\in A\times A\setminus C}\cap D(z)].
	\end{equation*}
	is a patch closed subset of $2^{(A \times A)}$ and therefore it is spectral.
	
	(2) Let $F$ be a finite set of generators of $C$. Then, the patch closed subset $D(F)\cap S_C$ of $2^{(A \times A)}$ gives the collection of all proper subcongruences of the congruence $\mathcal{C}$ and therefore it is spectral. The collection of all proper prime subcongruences of the congruence $C$ is also a spectral space since it is given by the following patch closed subset of $2^{(A \times A)}$
	\begin{equation*}
	D(F)\cap S_C \cap[\underset{a,b\in A\times A}\cap D(a\cdot_tb)\cup V(a)\cup V(b)].
	\end{equation*}
\end{proof}

\noindent In particular, Proposition \ref{proposition: subcongruences} implies that $\Spec_\mathfrak{c} A\subseteq 2^{(A \times A)}$ is a spectral space as follows. 

\begin{cor}
	Let $A$ be a semiring. 
	\begin{enumerate}
		\item 
		The collection of all congruences and the collection of all proper congruences of $A$ are spectral spaces with the hull-kernel topology.
		\item 
		If $A$ is an idempotent semiring, the collection of all prime congruences $\Spec_\mathfrak{c} A$ of $A$ is a spectral space with the hull-kernel topology.
	\end{enumerate}
\end{cor}
\begin{proof}
	The result follows from Proposition \ref{proposition: subcongruences} by taking $C=A\times A$ which is generated by the element $(1,0)\in A\times A$.
\end{proof}

\iffalse
\begin{mydef}\label{definition: closure for congruence}
	Let $A$ be a semiring, $C$ a congruence on $A$, and $S_C$ the set of subcongruences of $C$. A closure operation $cl$ on $S_C$ is a set map:
	\[
	cl:S_C \to S_C, \quad D\mapsto D^{cl}
	\] 	
	which satisfies the following:
	\begin{enumerate}
		\item 
		(Extension) $D \subseteq D^{cl}$ for all $D \in S_C$. 
		\item 
		(Idempotence) $D^{cl}=(D^{c})^{cl}$ for all $C \in S_C$. 
		\item 
		(Order-preservation) If $D_1 \subset D_2$, then $D_1^{cl} \subseteq D_2^{cl}$ for all $D_1,D_2 \in S_C$. 
	\end{enumerate}
\end{mydef}
\fi
\begin{mydef}
	Let $A$ be a semiring, $C$ a congruence on $A$, and $S_C$ the poset of all subcongruences of $C$. A closure operation $c:S_C \to S_C$ is said to be \emph{of finite type} if for any $D\in S_C$, 
	\[
	D^c=\bigcup\{E^c \mid E \subseteq D,E \in S_C,\textrm{ $E$ is finitely generated}\}.
	\]
\end{mydef}
 We will discuss examples of closure operation on $S_{C}$ in Section \ref{Cl by cong}. We will now show that given a finite type closure operation on $S_C$, the collection of all subcongruences of $C$ which remain fixed under the closure operation, forms a spectral space. The proof of this is analogues to the proof of \cite[Proposition 3.4]{Fo}, but we include it here for completeness.
\begin{pro}\label{proposition: of fnite type}
	Let $A$ be a semiring and $C$ be a congruence on $A$. Let $c$ be a closure operation of finite type on $S_C$ (as in Proposition \ref{proposition: subcongruences}). Then the following set
	\[
	X:=\{D \in S_C \mid D^c=D\}
	\]
	is a spectral space. 
\end{pro}
\begin{proof}
Let $\mathcal{U}$ be an ultrafilter on $X$ and $\mathbb{S}$ is the subbasis of $X$ (induced from the hull-kernel topology of $2^{A \times A}$). By Theorem \ref{2.2}, it is enough to prove that
\begin{equation*}
X_{\mathbb{S}}({\mathcal{U}}):= \{x \in X\ |\ [\forall S\in \mathbb{S}, x \in S \iff S \in \mathcal{U}]\} \neq \emptyset.
\end{equation*}
Consider the set $D_{\mathcal{U}}:=\{a \in C \mid V(a)\cap X\in \mathcal{U}\}$. Since $\mathcal{U}$ is an ultrafilter, we have 
\[
D(a) \in \mathcal{U} \iff V(a) \not \in \mathcal{U} \iff D_\mathcal{U} \not \in V(a) \iff D_\mathcal{U} \in D(a) 
\]
Thus, to show  that $D_{\mathcal{U}} \in X_{\mathbb{S}}({\mathcal{U}})$, it is enough to prove $D_{\mathcal{U}}^c \subseteq D_{\mathcal{U}}$. Suppose $a \in D_{\mathcal{U}}^c$. Since our closure operation $c$ is of finite type, there is a finitely generated congruence $D' \subseteq D_{\mathcal{U}}$ such that $a \in (D')^c$. It follows that $a \in F^c$ for any congruence $F$ containing $D'$. Therefore, if $D'$ is generated by $\{a_1,...,a_n\}$, then we have
\begin{equation}\label{inclusion}
\cap_{i=1}^n V(a_i)\cap X = V(D') \cap X \subseteq V(a) \cap X
\end{equation}
As $a_i \in D_\mathcal{U}$, we have $V(a_i) \cap X\in \mathcal{U}$. Since $\mathcal{U}$ is an ultrafilter, it follows from \eqref{inclusion} that $V(a) \cap X \in \mathcal{U}$.

\end{proof}

\section{Valuations, valuation orders, and prime congruences}

In this section, by appealing to results in \cite{tolliver2016extension} by the third author, we prove that for an idempotent semiring $A$, there is a bijection between the space of valuations on $A$ and the set of prime congruences on $A$. We then prove that indeed the space of valuations on $A$ is a spectral space, which is analogous to the fact that adic spaces are spectral.

For a totally ordered abelian group $(\Gamma,+_\Gamma)$, following the notation of \cite{tolliver2016extension}, we let $\Gamma_{\max}$ be the semifield with the underlying set $\Gamma \cup \{-\infty\}$ together with the following addition and multiplication: for $x,y \in \Gamma$,
\begin{equation}
x+y:= \max\{x,y\}, \qquad xy:= x+_\Gamma y
\end{equation}
with $ x+(-\infty) = x=(-\infty) +x$ and $x (-\infty) = (-\infty) = (-\infty) x$. For instance, when $\Gamma=(\mathbb{R},+)$, $\Gamma_{\max}$ is the tropical semifield. Now, we recall the definition of a valuation on an idempotent semiring. For the notational convenience, for $\Gamma_{\max}$, we will just write $1$ for the multiplicative identity and $0$ for the additive identity of $\Gamma_{\max}$.

\begin{mydef}\cite[Definition 2.2.]{izhakian2011glimpse}\label{definition: valuation}
Let $A$ be an idempotent semiring. By a valuation on $A$, we mean a function $\nu:A \to \Gamma_{\max}$ for some totally ordered abelian group $\Gamma$ satisfying the following properties. 
\begin{enumerate}
	\item[(a)]
$\nu(0)=0$ and $\nu(1)=1$.  
	\item[(b)] 
$\nu(xy) = \nu(x)\nu(y)$ $\forall x,y \in A$. 
\item[(c)]\label{equation: strict valuation}
$\nu(x+y)=\nu(x)+\nu(y) ~(= \max \{\nu(x),\nu(y)\})$. 
\end{enumerate}
\end{mydef}

\begin{rmk}
As it was pointed out in \cite{tolliver2016extension} (in the paragraph right after Definition 1.2), the following two conditions together are equivalent to (c):
\begin{enumerate}
	\item[(c1)] 
$\nu(x+y) \leq \nu(x) +\nu(y)$ $\forall x,y \in A$. 
\item[(c2)]
$\nu(x) \leq \nu(x+y)+\nu(y)$ $\forall x,y \in A$. 
\end{enumerate}
Strictly speaking, as there may be an element $a \neq 0 \in A$ such that $\nu(a)=0$, we should call $\nu$ a semivaluation in Definition \ref{definition: valuation}. However, we will just call $\nu$ a valuation so that our terminology is compatible with \cite{tolliver2016extension}. 
\end{rmk}

\begin{rmk}
A valuation as in the above definition means simply a homomorphism from $A$ to $\Gamma_{\max}$. In fact, a function $\nu:A \to \Gamma_{\max}$ satisfying the conditions in Definition \ref{definition: valuation} is first introduced in \cite{izhakian2011glimpse}, where the authors called it a \emph{strict valuation}. Some properties of strict valuations were studied in \cite{jun2018valuations} in connection to tropical geometry. 
\end{rmk}

For a semiring $A$, and a multiplicative subset $S$ of $A$, one can define the localization $S^{-1}A$ as in the classical case. When $A$ is multiplicatively cancellative, we let $\Frac(A):=S^{-1}A$, where $S=A-\{0\}$. In this case, the canonical map $A \to \Frac(A)$ is an injection and $\Frac(A)$ is a semifield. We further recall the following standard definition. 

\begin{mydef}
Let $A$ be an idempotent semiring. Let $\nu_1$ and $\nu_2$ be valuations on $A$. We say that $\nu_1$ and $\nu_2$ are equivalent if there exists an isomorphism $f:\nu_1(A)\to \nu_2(A)$ of semirings such that the following diagram commutes:
\begin{equation}
\begin{tikzcd}[column sep=small]
& A \arrow[dl,swap, "\nu_1"] \arrow[dr,"\nu_2"] & \\
\nu_1(A) \arrow[swap]{rr}{f} & & \nu_2(A)
\end{tikzcd}
\end{equation}
We let $\Spv A$ be the set of the equivalences of valuations on $A$. 
\end{mydef}

It was shown in \cite[Proposition 4.11]{tolliver2016extension} that there is a split surjection from the set of valuations to the set of prime congruences.  Now, we prove that if we instead work with equivalence classes of valuations, this becomes a one-to-one correspondence between $\Spv A$ and the set $\Spec_{\mathfrak{c}}A$ of prime congruences on $A$.

\begin{pro}\label{pro: Spv and Spec}
Let $A$ be an idempotent semiring. Then we have a bijection of sets:
\[
\Spv A \simeq  \Spec_\mathfrak{c} A,
\]
where $\Spv A$ is the set of equivalence classes of valuations on $A$ and $\Spec_\mathfrak{c} A$ is the set of prime congruences on $A$. 
\end{pro}
\begin{proof}
Let $C$ be a prime congruence on $A$. It is proved in \cite{joo2014prime} that the quotient $A/C$ is multiplicatively cancellative and totally ordered. It follows that we have an injection $A/C \to \Frac(A/C)$ and $\Frac(A/C)$ is an idempotent semifield. In fact, as $A/C$ is totally ordered, $\Frac(A/C)$ is totally ordered and hence the following map
\[
\nu_C:A \longrightarrow \Frac(A/C), \quad a \mapsto \frac{[a]}{1},
\]	
where $[a]$ is the equivalence class of $a$ in $A/C$ is a valuation. 

Conversely, suppose that we have a valuation $\nu:A \to S=\Gamma_{\max}$. We claim that the following set:
\[
C_{\nu}:=\{(x,y) \in A \times A \mid \nu(x)=\nu(y)\}
\]
is a prime congruence on $A$. In fact, clearly $C_{\nu}$ is a congruence relation since $C_\nu$ is the kernel congruence of $\nu$. Furthermore, $A/C_\nu$ is isomorphic to $\nu(A)$ which is totally ordered and cancellative by being a subsemiring of $S$, which is totally ordered and cancellative. It follows again from the results in \cite{joo2014prime} that $C_\nu$ is prime. It is clear then from the definition that if $\nu_1$ and $\nu_2$ are valuations on $A$ which are equivalent, then $C_{\nu_1}=C_{\nu_2}$. Hence, we have two functions:
\[
f:\Spv A \longrightarrow \Spec_{\mathfrak{c}} A,\quad [\nu] \mapsto C_\nu,
\]
where $[\nu]$ is the equivalence class of a valuation $\nu$ in $\Spv A$, and 
\[ 
g: \Spec_{\mathfrak{c}} A \longrightarrow \Spv A, \quad C \mapsto [\nu_C].  
\]

All it remains to show is that $f$ and $g$ are inverses to each other. Let $C \in \Spec_{\mathfrak{c}} A$, then we have
\[
g(C): A \longrightarrow \Frac(A/C), \quad a \mapsto \frac{[a]}{1},
\]
where $[a]$ is the equivalence class of $a \in A$ in $A/C$. Now, $f\circ g(C)$ is the following congruence:
\[
(f\circ g)(C)=\{(x,y) \in A \times A \mid g(C)(x)=g(C)(y)\} =\{(x,y) \in A \times A \mid [x]=[y]\}=C,
\]
showing that $f\circ g$ is the identity on ${\Spec_{\mathfrak{c}} A}$. 

Finally, let $[\nu] \in \Spv A$ such that $\nu:A \to S$. Then, $f([\nu])$ is the following prime congruence:
\[
f([\nu])=\{(x,y) \in A \times A \mid \nu(x)=\nu(y)\}.
\]
Notice that $f([\nu])$ is the kernel congruence of $\nu$, we have that $A/f([\nu]) \simeq \nu(A)$. It follows that $g\circ f ([\nu])$ is a valuation defined as follows:
\[
g( f ([\nu])): A \longrightarrow \Frac(\nu(A)), \quad a \mapsto \frac{[a]}{1},
\]
where $[a]$ is the equivalence class of $a \in A$ in $\nu(A)$. But, with the injection $\nu(A) \hookrightarrow \Frac(\nu(A))$, we have that $[\nu]=[g(f ([\nu]))]$, showing that $g \circ f$ is the identity on $\Spv A$. 
\end{proof}

Now, we prove that $\Spv A$ is a spectral space. For this, we first use the notion of the space of valuation orders in \cite[\S 7]{tolliver2016extension}. Note that we do not give the original definition given in \cite{tolliver2016extension}, but rather an equivalent description proved in the same paper. 

\begin{mydef}
Let $A$ be an idempotent semiring. A preorder $\preceq$ on $A$ is said to be a \emph{valuation order}  if there exists a valuation $\nu:A \to \Gamma_{\max}$ such that
\[
x \preceq y \iff \nu(x) \leq \nu(y). 
\]
\end{mydef}

\begin{pro}\label{proposition: valuation order = space of valuations}
Let $A$ be an idempotent semiring. Then, there is a bijection of sets between $\Spv A$ and the set of valuation orders on $A$. 
\end{pro}
\begin{proof}
Let $X$ be the set of valuations on $A$ and $Y$ be the set of valuation orders on $A$. We define the function $f: X \to Y$, sending a valuation $\nu$ to the valuation order $\preceq_\nu$ defined by $\nu$, that is, $x \preceq_\nu y$ if and only if $\nu(x) \leq \nu(y)$. It follows from \cite[Corollary 7.7.]{tolliver2016extension} that $f$ is onto. Furthermore, it is clear that if $\nu_1$ and $\nu_2$ are two equivalent valuations, then the induced valuation orders $\preceq_{\nu_1}$ and $\preceq_{\nu_2}$ are the same. Hence, $f$ induces a surjective map $\tilde{f}:\Spv A \to Y$. 

Finally, we claim that $\tilde{f}$ is injective. From \cite[Proposition 7.6.]{tolliver2016extension}, it follows that for a given valuation order $\preceq$, if $\sim$ is the relation defined by $x\sim y$ if and only if $x\preceq y \preceq x$, then $A/\sim$ is a totally ordered cancellative idempotent semiring, and the canonical order on $A/\sim$ agrees with the one induced by $\preceq$. In particular, this determines an element in $\Spec_{\mathfrak{c}} A$. Suppose that $\nu_1$ and $\nu_2$ are valuations inducing the same valuation orders. By what we just described, if $\nu_1$ and $\nu_2$ determine the same valuation order then they will determine the same prime congruence on $A$. Therefore, $\nu_1$ and $\nu_2$ should be equivalent, showing that $\tilde{f}$ is an injection. 
\end{proof}

We will be using the following proposition to show that $\Spv A$ is a spectral space. One may also find more details for the case for rings in \cite{adicspace}.

\begin{pro}\cite[Proposition 3.31]{adicspace2} \label{proposition: Hochster criterion}
Let $X'=(X_0,\tau')$ be a quasi-compact topological space. Let $\mathcal{U} \subseteq \tau'$ be a collection of clopen subsets of $X'$. Let $\tau$ be the topology on $X_0$ generated by $\mathcal{U}$. If $X=(X_0,\tau)$ is $T_0$, then $X$ is a spectral space. 
\end{pro}

In \cite{tolliver2016extension}, the third author imposed the topology on the space of valuation orders which we recall now. First, we identify the space of valuation orders with a subset of $2^{A \times A}$ by identifying each valuation order $\preceq$ with the following subset of $A\times A$:
\[
S_\preceq:=\{(x,y) \in A \times A \mid x\preceq y\}.
\]

Each subset $S \subseteq A \times A$ can be considered as a function $f_S: A \times A \to \{0,1\}$, where $f_S(a)=1$ if and only if $a \in S$. In particular, we can identify $2^{A \times A}$ with the set of functions $f:A \times A \to \{0,1\}$ which in turn can be considered as follows:
\[
2^{A \times A} = \prod_{a \in A\times A}\{0,1\}^{(a)}, 
\]
that is a product of copies of $\{0,1\}$. Now, we impose the product topology on $\prod\limits_{a \in A\times A}\{0,1\}^{(a)}$ and then impose the subspace topology to the space of valuation orders, which can identified with $\Spv A$ by Proposition \ref{proposition: valuation order = space of valuations}, hence $\Spv A$ becomes a topological space. We let $\tau'$ be this topology. 

\begin{rmk}
We note that the above topology is different from the hull-kernel topology on $A \times A$. 
\end{rmk}

To apply Proposition \ref{proposition: Hochster criterion}, we impose another topology on $\Spv A$ which is analogous to the case of rings. Let $\mathcal{U}$ be the sets of the form: for $(x,y) \in A \times A$, 
\begin{equation}\label{eq: open}
D(x,y)=\{[\nu] \in \Spv A \mid \nu(x) \leq \nu(y),\nu(y) \neq 0\}.
\end{equation}
Let $\tau$ be the topology on $\Spv A$ with a basis of open subsets $\{D(x,y)\}_{(x,y) \in A \times A}$. 

\begin{pro}\label{pro: spectral1}
With the notation as above, $(\Spv A, \tau)$ is a spectral space. 
\end{pro}
\begin{proof}
First, by \cite[Proposition 7.4]{tolliver2016extension}, we know that $(\Spv A, \tau')$ is quasi-compact. One further notices that in $\prod\limits_{a \in A\times A}\{0,1\}^{(a)}$, for any $(x,y) \in A \times A$, the following set
\[
C(x,y):=\{Z \in \prod\limits_{a \in A\times A}\{0,1\}^{(a)} \mid Z_{(x,y)}=1 \}
\]
is clopen with the topology $\tau'$ since one only has to look at the coordinate $(x,y)$. Furthermore, we have that
\[
C(x,y) \cap \Spv A=\{[\nu] \in \Spv A \mid \nu(x) \leq \nu(y)\}. 
\]
In particular, 
\[
D(x,y)=C(x,y) \cap (\Spv A - C(y,0)), 
\]
where $D(x,y)$ is an open subset of $(\Spv A, \tau)$ as in \eqref{eq: open}. This proves that $D(x,y)$ is clopen with the topology $\tau'$. 

Next, suppose that $[\nu_1],[\nu_2] \in \Spv A$ are topologically indistinguishable. For any $x \in A$, we have that
\[
\nu_1(x) \neq 0 \iff \nu_1 \in D(x,x) \iff \nu_2 \in D(x,x) \iff \nu_2(x) \neq 0.
\]
In particular, $\nu_1$ and $\nu_2$ have the same kernel. Now, for any $x,y \in A \times A$, we have that
\[
\nu_1(x) \leq \nu_1(y) \iff x,y \in \ker(\nu_1) \textrm{ or } \nu_1 \in D(x,y)
\]
Since $\nu_1$ and $\nu_2$ have the same kernel and they are topologically indistinguishable, we further have that 
\[
x,y \in \ker(\nu_1) \textrm{ or } \nu_1 \in D(x,y) \iff x,y \in \ker(\nu_2) \textrm{ or } \nu_2 \in D(x,y) \iff \nu_2(x) \leq \nu_2(y). 
\]
Therefore, $\nu_1(x) \leq \nu_1(y)$ if and only if $\nu_2(x) \leq \nu_2(y)$ for any $x,y \in A$, showing that $[\nu_1]=[\nu_2]$ since $\nu_1$ and $\nu_2$ induce the same valuation order. 

Finally, we conclude that $(\Spv A, \tau')$ is quasi-compact, $\mathcal{U}:=\{D(x,y)\}_{(x,y) \in A \times A}$ is a collection of clopen subsets of $\Spv A$ with respect to the topology $\tau'$, and $(\Spv A, \tau)$ is $T_0$. Hence, it follows from Proposition \ref{proposition: Hochster criterion} that $(\Spv A, \tau)$ is a spectral space. 
\end{proof}

\begin{rmk}\label{remark: rem}
	Here are two final remarks. 
	\begin{enumerate}
		\item
		Our bijection in Proposition \ref{pro: Spv and Spec} is analogous to the case of commutative rings. For a commutative ring $R$, one has the following function
		\[
		\ker: \Spv R \longrightarrow \Spec R, \quad \nu \mapsto \ker (\nu). 
		\]
		The function $\ker$ is continuous and quasi-compact map, but the map $\ker$ is not a bijection in general. For instance, if $R = \mathbb{Q}$, then we have $\Spv R = \Spec \mathbb{Z}$ (as sets) whereas $\Spec \mathbb{Q} = \{0\}$. 
		\item 
		Let $A=\mathbb{T}[x_1,...,x_n]$. 
		Mincheva showed in her thesis \cite{mincheva2016semiring} that the subset of $\Spec_\mathfrak{c} A$ consisting of all geometric congruences\footnote{A geometric congruence simply means a congruence $C$ on $A$ such that $A/C \simeq \mathbb{T}$. In particular, they are prime congruences. } could provide a way to interpret the set-theoretic tropicalization in the setting of prime congruences. 
	\end{enumerate}
\end{rmk}

\section{Closure operations for semirings}\label{section: closure operations for semirings}

In this section, we explore closure operations for semirings. We investigate closure operations on the set of ideals and the set of congruences on an idempotent semiring $A$. 
If $\mathcal{I}$ denotes the poset of all ideals of an idempotent semiring $A$, then recall that (Definition \ref{clo-op}) a closure operation on $\mathcal{I}$ is a map
$cl: \mathcal{I}\longrightarrow \mathcal{I}, I\mapsto I^{cl}$ which satisfies extension, order-preservation and idempotence properties.
%satisfying the following conditions:
%	\begin{enumerate}
%		\item 
%	(Extension) $I \subseteq I^{cl}$ for all $I \in \mathcal{I}$.
%	\item 
%	(Order-preservation) If $I_1 \subseteq I_2$, then $I_1^{cl} \subseteq I_2^{cl}$ for all $I_1,I_2 \in \mathcal{I}$.
%	\item 
%	(Idempotence) $I^{cl}=(I^{cl})^{cl}$ for all $I \in \mathcal{I}$.
%	\end{enumerate}	
Some simple examples of closure operations on $\mathcal{I}$ are: 
\begin{enumerate}
	\item 
(Identity closure) $I^{cl}:= I$ for all $I \in \mathcal{I}$.
\item 
(Indiscrete closure) $I^{cl}:= A$ for all $I \in \mathcal{I}$.
\item 
(Radical closure) $I^{cl}:= \sqrt{I}= \{x\in A \mid \exists n \in \mathbb{N} \text{ such that }x^n \in I\}$ for all $I \in \mathcal{I}$.
\end{enumerate}

We now introduce some interesting closure operations on an idempotent semiring like closure at a congruence, integral closure and Frobenius closure. We also recall the closure operation from \cite[Theorem 3.7]{les2} which we call {\it $k$-closure operation}. It can be easily verified that all the closure operations that we discuss in this section are closure operations of finite type. Closure operations like integral and Frobenius are named so because they are inspired by similar operations for rings in classical commutative
algebra.
\subsection{$k$-closure}
The kernel of a congruence is just the equivalence class of the 0 element i.e., $\Ker(C) = \{ a \in A \mid (a,0) \in C\}$.
It is easy to see that the kernel of a congruence is always an ideal. It follows from \cite[Proposition 2.2 (iii)]{joo2014prime} that $(a+b, 0) \in C$ implies $(a, 0), (b,0)\in C$. Hence, $\Ker(C)$ is in fact a $k$-ideal.

We now briefly discuss $k$-closure operation. For details, see \cite[Theorem 3.7]{les2}. Let $A$ be an idempotent semiring. For an ideal $I$ of $A$, let $ C_I$ denote the corresponding congruence i.e.,
\[
C_I= \langle\{ (a,0) \in A\times A \mid a \in I\}\rangle = \{ (x,y) \in A\times A \mid \exists\  z \in I \text{ such that } x+z = y+z\}.
\]
It can be shown that $\Ker(C_I)$ is the unique smallest $k$-ideal containing $I$. In other words, $\Ker(C_I)$ is the $k$-closure of $I$ i.e., $\Ker(C_I) = \{ x\in A \mid \exists\ z\in I \text{ such that } x+z = z\}$
and 
\begin{equation}\label{Ker}
I \mapsto \Ker(C_I)
\end{equation}
defines a closure operation on the collection of all ideals of $A$ which is called the $k$-closure operation. 

The integral closure operation, which we will introduce soon, is essentially a generalization of the $k$-closure operation.
\subsection{Closure with respect to a congruence}
Let $C$ be a congruence on an idempotent semiring $A$. Let $\mathcal{I}$ be the set of all ideals of $A$. For an ideal $I\in \mathcal{I}$, define
\begin{equation}
I^{C}:=\{x\in A~ \mid ~\exists~ z\in I \text{ such that } (x,z)\in C\}.
\end{equation}
For $x\in I^{C}$ and $a\in A$, we have $(ax,az)=(a,a)(x,z)\in C$. Since $az \in I$, it implies that $ax\in I^{C}$. It follows that $I^{C}\in \mathcal{I}$. Extension and order-preservation properties are obvious and therefore we also have $I^{C}\subseteq ({I^{C}})^{C}$. Using transitivity of $C$ it can be easily verified that $(I^{C})^{C}\subseteq I^{C}$. Hence, $I \mapsto I^{C}$ defines a closure operation on $\mathcal{I}$.

A slight variant of the above operation also gives a closure operation as we explain now. Given a congruence $C$, for an ideal $I\in \mathcal{I}$, define
\begin{equation}\label{relation with int}
I^{[C]}:=\{x\in A~ \mid ~\exists~ z\in I \text{ and } y\geq x \text{ such that } (y,z)\in C\}
\end{equation}
Equivalently, we have 
\begin{equation}\label{cl at C}
I^{[C]}=\{x\in A~ \mid ~\exists~ z\in I \text{ such that } (z, x+z)\in C\}=\{x\in A~ \mid ~\exists~ z\in I \text{ such that } [x]\leq [z]\in A/C\}
\end{equation} 
We only check the idempotence condition:
\begin{align*}
(I^{[C]})^{[C]} &= \{x\in A~ \mid ~\exists~ z\in I^{[C]} \text{ such that } [x]\leq [z]\}\\
&= \{x\in A~ \mid ~\exists~ z\in A \text{ and } y\in I \text{ such that } [x]\leq [z]\leq [y]\}\\
&= \{x\in A~ \mid ~\exists~ y\in I \text{ such that } [x]\leq [y]\} = I^{[C]}
\end{align*}
It could be interesting to note that $I^{[C]}$ is a $k$-ideal even if $I$ is not.

\subsection{Integral closure}
For an ideal $I$ of an idempotent semiring $A$, define
 \begin{equation}
 I^{int}:=\{x\in A~ \mid ~x^n+ a_1x^{n-1} + \hdots + a_n = b_1x^{n-1} + \hdots + b_n \text{ for some } n \in \mathbb{N} \text{ and } a_i,b_i\in I^{i}\}
  \end{equation}
We first give an useful interpretation of the set $I^{int}$.  
\begin{lem}
	Let $I$ be an ideal of an idempotent semiring $A$. Then, 
	\begin{equation}\label{ new intprt1}
		I^{int}=\{x\in A~ \mid \exists~ z \in I \text{ such that }(x+z)^n = z(x+z)^{n-1} \text{ for some } n \in \mathbb{N}\}.
	\end{equation}
\end{lem}

\begin{proof}
	Let $x\in A$ be such that $(x+z)^n = z(x+z)^{n-1}$ for some $z\in I$ and $n\in \mathbb{N}$. Equivalently, we have 
	$$	x^n + z x^{n-1} + \hdots + z^n =(x+z)^n=  z(x+z)^{n-1}= z x^{n-1} + \hdots + z^n$$ with $z^i\in I^i$ and this shows that $x\in I^{int}$.\\
	Conversely, let $x\in  I^{int}$. Then, there exists some $n \in \mathbb{N} \text{ and } a_i,b_i\in I^{i}$ such that 
	\begin{equation}\label{intdef}
		x^n+ a_1x^{n-1} + \hdots + a_n = b_1x^{n-1} + \hdots + b_n 
	\end{equation}
For each $i\in \{1,\hdots, n\}$,
	there exists $n_i,m_i\in \mathbb{N}$ such that
	\begin{align}\label{sumofprod}
		a_i=\sum_{k=1}^{n_i}\prod_{j=1}^i {a_{ij}}_{k}\qquad\qquad b_i=\sum_{k=1}^{m_i}\prod_{j=1}^i {b_{ij}}_{k}
	\end{align}
	for some ${a_{ij}}_{k}, {b_{ij}}_{k}\in I$. Let
	\begin{equation}\label{2sum}
		\sigma_{a_i}:= \sum {a_{ij}}_{k}\qquad\qquad  	\sigma_{b_i}:= \sum {b_{ij}}_{k}
	\end{equation}
 where the summations respectively run over all the ${a_{ij}}_{k}$ and ${b_{ij}}_{k}$ appearing in \eqref{sumofprod}.
	Let
	\begin{equation}\label{1sum}
		z:=\sum_{i=1}^n\sigma_{a_i} + \sigma_{b_i}
	\end{equation}
	Clearly, we have $z\in I$.  It follows from \eqref{2sum} and \eqref{1sum} that 
	\begin{equation}
		{a_{ij}}_{k}\leq \sigma_{a_i}\leq z\qquad\qquad 	{b_{ij}}_{k}\leq \sigma_{b_i}\leq z
	\end{equation}
Therefore, by \eqref{sumofprod} we have $a_i\leq z^i$ and $b_i\leq z^i$ for all $i\in\{1,\hdots, n\}$. In other words, $a_i+z^i=z^i$ and $b_i+z^i=z^i$ for all $i\in\{1,\hdots, n\}$. 
	Thus, adding $z x^{n-1} + \hdots + z^n$ to both sides of the equation \eqref{intdef}, we obtain 
	\begin{equation}\label{ new intprt}
		(x+z)^n=x^n + z x^{n-1} + \hdots + z^n = z x^{n-1} + \hdots + z^n  = z(x+z)^{n-1}
	\end{equation}
	This completes the proof.
\end{proof}

To show that $I^{int}$ is indeed an ideal of $A$, we will use a binomial expansion trick for idempotent semiring which we prove next. 

\begin{lem}\label{binom trick}
Let $A$ be an idempotent semiring. For any $a, b\in A$ and $m, n\in \mathbb{N}$, we have 
\begin{itemize}
\item[(1)] $(a + b)^{m + n} = a^m(a + b)^n + b^n(a + b)^m$
\item[(2)] If $A$ is also cancellative, we have $(a+b)^n = a^n + b^n$.
\end{itemize}
\end{lem}
\begin{proof}
We have that
\[
(a+b)^{m+n} = \sum_{k=0}^{m+n} a^kb^{m+n-k} = \sum_{k=0}^m a^kb^{m+n-k} + \sum_{k=m+1}^{m+n}a^kb^{m+n-k}
\]
\[
=\sum_{k=0}^m a^kb^{m+n-k} + a^mb^n +  \sum_{k=m+1}^{m+n}a^kb^{m+n-k} = \sum_{k=0}^m a^kb^{m+n-k} + \sum_{k=m}^{m+n}a^kb^{m+n-k}
\]
\[
=b^n\sum_{k=0}^ma^kb^{m-k} + \sum_{j=0}^{n} a^{m+j}b^{n-j} = b^n\sum_{k=0}^ma^kb^{m-k} + a^m\sum_{j=0}^{n} a^{j}b^{n-j}=b^n(a+b)^m + a^m(a+b)^n. 
\]
This proves (1).

Now, assume $A$ is cancellative. Putting $m=n$ in (1), we obtain
\[
(a + b)^{n + n} = a^n(a + b)^n + b^n(a + b)^n = (a^n +b^n)(a+b)^n.
\]
Since $A$ is idempotent (and hence strict) and cancellative, this gives us $(a+b)^n = a^n + b^n$.
\end{proof}
\begin{lem}
Let $A$ be an idempotent semiring. For any ideal $I$ of $A$, 
\begin{equation}\label{int}
I^{int}:=\{x\in A~ \mid \exists~ z \in I \text{ such that }(x+z)^n = z(x+z)^{n-1} \text{ for some } n \in \mathbb{N}\}
\end{equation}
is an ideal of $A$.
\end{lem}
\begin{proof}
Let $x, y \in I^{int}$. Then, there exists $z, w \in I$ and $n, m \in \mathbb{N}$ such that  $(x+z)^n = z(x+z)^{n-1}$ and $(y+w)^m = w(y+w)^{m-1}$. This implies $x^n \leq z(x+z)^{n-1}$ and $y^m \leq w(y+w)^{m-1}$.  Let $s = x + y + z + w$.  Then, using the binomial trick (Lemma \ref{binom trick}(1)) we have  
\begin{align}\label{sum1}
(x+y)^{m+n-1} &= x^n(x+y)^{m-1} + y^m(x+y)^{n-1} \leq x^ns^{m-1} + y^ms^{n-1}\nonumber\\
&\leq z(x+z)^{n-1}s^{m-1} + w (y+w)^{m-1}s^{n-1}\nonumber\\
&\leq (z+w)s^{m+n-2} = (z+w)(x+y+z+w)^{m+n-2}
\end{align} 
Multiplying $s$ on either sides of \eqref{sum1}, we have $$(x+y)^{m+n-1}(x+y+z+w)\leq(z+w)(x+y+z+w)^{m+n-1}$$
In other words, we have 
\begin{equation}\label{sum2}
(z+w)(x+y+z+w)^{m+n-1} +(x+y)^{m+n-1}(x+y+z+w) = (z+w)(x+y+z+w)^{m+n-1}
 \end{equation}
Using the binomial trick (Lemma \ref{binom trick}(1)) again, we obtain from \eqref{sum2}
$$(x+y+z+w)^{m+n}=(z+w)(x+y+z+w)^{m+n-1}$$
Since $z+w\in I$, this implies $x+y \in I^{int}$. Also, for $x\in I^{int}$ it can be easily seen that $ax\in I^{int}$ for any $a\in A$. It follows that $I^{int}$ is an ideal of $A$. 
\end{proof}

Let $I'$ denote the $k$-closure of an ideal $I\in \mathcal{I}$. In other words, $I'$ is the intersection of all $k$-ideals containing $I$. 
\begin{lem}\label{int-sat}
Let $A$ be an idempotent semiring. For any ideal $I$ of $A$, we have $I^{int}= (I')^{int}$.
\end{lem}
\begin{proof}
Let $x \in (I')^{int}$. Then, by definition, there exists $y\in I'$ with $(x+y)^n = y(x+y)^{n-1}$. Since $y\in I'$, there exists some $z\in I$ such that $y\leq z $. Therefore, we have
\begin{equation}\label{sat-int}
x^n\leq (x+y)^n = y(x+y)^{n-1}\leq z(x+z)^{n-1} 
\end{equation}
 Using the binomial trick (Lemma \ref{binom trick}(1)), it follows from \eqref{sat-int}  $$z(x+z)^{n-1} = z(x+z)^{n-1}+x^n = z(x+z)^{n-1}+ x^{n-1}(x+z) = (x+z)^n$$
Thus, $x \in I^{int}$. The other inclusion $I^{int} \subseteq (I')^{int}$ is obvious.
\end{proof}
\begin{pro}\label{pro: integral closure}
Let $A$ be an idempotent semiring. Let $ \mathcal{I}$ denote the collection of all ideals of $A$. For any ideal $I\in \mathcal{I}$,
\begin{equation}
I \mapsto (I^{int})'
\end{equation}
defines a closure operation on $ \mathcal{I}$ which we call the integral closure operation.
\end{pro}
\begin{proof}
Clearly $I \subseteq I^{int} \subseteq (I^{int})'$. Also, for any $I, J\in \mathcal{I}$ with $I \subseteq J$, it can be easily verified that $I^{int} \subseteq J^{int}$. Consequently, we also have  $(I^{int})' \subseteq (J^{int})'$. Let us now verify the idempotence condition. Let $x\in ({I^{int}})^{int}$. Then, there exists $z \in I^{int}$ such that $(x+z)^n = z(x+z)^{n-1}$ for some $n \in \mathbb{N}$. Also, for $z$ in $I^{int}$ there exists $w$ in $I$ such that $(z+w)^m= w(z+w)^{m-1}$ for some $m \in \mathbb{N}$. Now,
\begin{align*}
(x+z+w)^{nm} &= w(x+z+w)^{nm-1} + (x+z)^{nm-1}(x+z+w)\qquad \text{ (using (Lemma \ref{binom trick}(1)))}\\
             & = w(x+z+w)^{nm-1} + (x+z)^{nm} \qquad\qquad \text{( rest gets absorbed in the first term)}\\
             & = w(x+z+w)^{nm-1} + z^m(x+z)^{nm-m}\\
             & \leq  w(x+z+w)^{nm-1} + (z+w)^m(x+z)^{nm-m}\\
             & = w(x+z+w)^{nm-1} + w(z+w)^{m-1}(x+z)^{nm-m}\\
             & = w(x+z+w)^{nm-1} \qquad\qquad\qquad\qquad\qquad \text{(rest gets absorbed in the first term)}
\end{align*}
In other words, we have $w(x+z+w)^{nm-1} = w(x+z+w)^{nm-1}+(x+z+w)^{nm} = (x+z+w)^{nm}$. This shows that $x+z \in I^{int}$. Since $z$ is also in $I^{int}$, this implies $x\in (I^{int})'$. Hence, $(I^{int})^{int}\subseteq (I^{int})'$. By Lemma \ref{int-sat}, we also have 
$(I^{int})' \subseteq ({(I^{int})}')^{int}= (I^{int})^{int}$. Thus,$(I^{int})^{int}= (I^{int})'$. Again applying Lemma \ref{int-sat}, we obtain
$$((I^{int})')^{int} = (I^{int})^{int}= (I^{int})'$$
Therefore, we have $(((I^{int})')^{int})' = (I^{int})'$. This completes the proof.
\end{proof}

It can be easily seen from \eqref{Ker} and \eqref{int} that for an idempotent cancellative semiring $A$, integral closure is same as $k$-closure. The following is an example in this context. 

\begin{myeg}
Consider the semiring $\mathbb{T}[x]$ of polynomials with coefficients in the tropical semifield $\mathbb{T}$. One can impose a congruence relation on $\mathbb{T}[x]$ as follows:
\[
f(x) \sim g(x) \iff f(\alpha)=g(\alpha)~ \forall \alpha \in \mathbb{T}. 
\]
Let $A:=\mathbb{T}[x]/\sim$. One may observe that $A$ is multiplicatively cancellative and strict (see \cite{jun2018valuations}). Hence for any ideal $I \subseteq A$, we have that
\[
I^{int}=\{x \in A \mid x+z=z, ~\textrm{ for some } z \in I\}. 
\]
In other words, in this case, $I^{int}$ is just the $k$-closure of $I$. For example, the ideal $I=\angles{\bar{x}}$ generated by $\bar{x}$ is a $k$-ideal. In fact, suppose that for $\overline{f(x)} \in A$, and 
\[
\overline{h(x)} \leq \overline{f(x)}=\overline{g(x)}\cdot \overline{x}. 
\]
That is 
\begin{equation}\label{eq: saturation}
\overline{h(x)+f(x)} = \overline{h(x)}+\overline{f(x)} = \overline{f(x)}.
\end{equation}
But, from \cite[Lemma 4.5]{jun2018valuations}, any $\overline{g(x)}$ has the factor $\overline{x}$ if and only if any representative of $\overline{g(x)}$ has no constant term. Since $\overline{f(x)} \in \angles{\bar{x}}$, we have that $h(x)+f(x)$ and $f(x)$ have no constant term, and in particular, $h(x)$ has no constant term. This shows that $\overline{h(x)}$ has the factor $\overline{x}$, and hence $\overline{h(x)} \in \angles{\bar{x}}$, showing that $\angles{\bar{x}}$ is a $k$-ideal. Hence $I^{int}=\angles{\bar{x}}$. 
\end{myeg}

For any semiring $A$, note that $\alpha\cdot_t\beta = \beta\cdot_t\alpha$ for any $\alpha, \beta \in A\times A$ where $\cdot_t$ denotes the twisted product. Using this it can be easily shown that in an idempotent semiring $A$, we have
\[
( \alpha + \beta)^n = \sum_{i=0}^n\alpha^i\beta^{n-i},
\]
where product means twisted product. We now point out an interesting connection between the integral closure operation and the closure operation with respect to a congruence. In what follows, all products are twisted product unless otherwise stated. 

Let $C=\cap P$ denote the congruence which is the intersection of all prime congruences $P$ of an idempotent semiring $A$. Consider the closure operation \eqref{cl at C} defined by $$I \mapsto I^{[C]}= \{x\in A~ \mid ~\exists~ z\in I \text{ such that } (z, x+z)\in C\}.$$
If $x$ is in $I^{[C]}$ , we have $(z, x+z) \in \cap P$ for some $z$ in $I$. By \cite[Theorem 3.9]{joo2014prime}, there exists $n,l \in \mathbb{N}$ and $c \in A$ such that, $((x+z)^n+c,0)(z,x+z)^l\in \textrm{Diag}(A)$. We will now expand the term $(z,x+z)^l$.
%We claim that
%\begin{align*}
%&(z,x+z)^l = (z(x+z)^{l-1},(x+z)^l)\qquad\qquad\text{when $l$ is odd }\\
%\text{and } &(z,x+z)^l = ((x+z)^l, z(x+z)^{l-1})\qquad\qquad\text{when $l$ is even }
%\end{align*}
For this, first note that for any $b \in A$, we have $(0,b)^r = (0, b^r)$ when $r$ is odd and $(0,b)^r = (b^r, 0)$ when $r$ is even. Also, for any $a \in A$, we have $(a,a)^n = (a^n,a^n)$ for any $n \in \mathbb{N}$ and $(a,a)(b,0)= (ab,ab)=(a,a)(0,b)$.
Using this, we obtain
\begin{align*}
(z,x+z)^l = ((z,z)+(0,x))^l &= \sum_{i=0}^l(z,z)^i(0,x)^{l-i}\\
&=\begin{cases}
      \sum_{i=1}^l(z^ix^{l-i},z^ix^{l-i}) + (0,x^l) & \text{when $l$ is odd}\\
      \sum_{i=1}^l(z^ix^{l-i},z^ix^{l-i}) + (x^l,0) & \text{when $l$ is even}
    \end{cases}   \\
&=\begin{cases}
      (z(x+z)^{l-1},(x+z)^l) & \text{when $l$ is odd}\\
     ((x+z)^l, z(x+z)^{l-1}) & \text{when $l$ is even}
    \end{cases}    
\end{align*}
Therefore, $x\in I^{[C]}$ implies that there exists $n,l \in \mathbb{N}$ such that either
\begin{align}\label{exp}
&((x+z)^n+c,0)(z,x+z)^l
=((x+z)^n+c,0)(z(x+z)^{l-1},(x+z)^l)\nonumber\\ &= (z(x+z)^{n+l-1},(x+z)^{n+l})
+ c(z(x+z)^{l-1},(x+z)^l)) \in \textrm{Diag}(A)
\end{align}
or 
\begin{align}\label{exp1}
&((x+z)^n+c,0)(z,x+z)^l
=((x+z)^n+c,0)((x+z)^l,z(x+z)^{l-1})\nonumber\\ &= ((x+z)^{n+l},z(x+z)^{n+l-1})
+ c((x+z)^l,z(x+z)^{l-1})) \in \textrm{Diag}(A)
\end{align}
Clearly, it follows from \eqref{ new intprt1} that any $x \in I^{int}$ satisfies a condition of the form \eqref{exp} or \eqref{exp1}. Therefore, $I^{int} \subseteq I^{[\cap P]}$ for any ideal $I$ of $A$.
\begin{rmk}
The observation $I^{int} \subseteq I^{[\cap P]}$ determines an idempotent semiring analogue (in fact, a stronger version) of \cite[Proposition 6.8.10]{HS}. To see this, observe that it follows from \eqref{cl at C} and Proposition \ref{pro: Spv and Spec} that
\begin{align*}
I^{[\cap P]} &= \{x\in A~ \mid ~\exists~ z\in I \text{ such that } [x]\leq [z]\in A/[\cap P]\}\\
&=\{x\in A~ \mid ~\exists~ z\in I \text{ such that } v(x)\leq v(z) \text{ for all valuations $v$ on $A$}\}
\end{align*}
This also shows how the elements in $I^{int}$ are related to the valuations on $A$.
\end{rmk}
\begin{rmk}
We have $I^{int} \subseteq I^{[\cap P]}$ for any ideal $I$ of an idempotent semiring $A$. If we further assume that $A$ is also idealic i.e., $a \leq 1$ for all $a \in A$,  it follows from \eqref{exp} that
\begin{align*}
&(z(x+z)^{n+l-1},(x+z)^{n+l})+ c(z(x+z)^{l-1},(x+z)^l)) + ((x+z)^{l-1}, (x+z)^{l-1}) \\ 
&= (z(x+z)^{n+l-1},(x+z)^{n+l})+((cz+1)(x+z)^{l-1},(1+x+z)(x+z)^{l-1}))\\
&= (z(x+z)^{n+l-1},(x+z)^{n+l})+((x+z)^{l-1},(x+z)^{l-1}))
\in \textrm{Diag}(A)
\end{align*}
In other words, the second summand of \eqref{exp} can be appropriately modified to make it belong in $Diag(A)$. Similarly, \eqref{exp1} can also be modified. Of course, this does not give us $z(x+z)^{n+l-1}=(x+z)^{n+l}$. But, in spirit, it is similar to \eqref{ new intprt1} and therefore to \eqref{ new intprt}. Considering the similarity of \eqref{ new intprt} to the classical ring theoretic notion of integral closure, $I \mapsto I^{[\cap P]}$ could also be a possible candidate for the notion of integral closure on idempotent idealic semirings.
\end{rmk}

\subsection{Frobenius closure}\label{subsection: frobenius closure}

Recall from Lemma \ref{binom trick} (2) that if $A$ is an idempotent cancellative semiring, then $f_n: A\longrightarrow A,  a \mapsto a^n$ defines an endomorphism of $A$ for any $n \in \mathbb{N}$. 

\begin{pro}
Let $A$ be an idempotent cancellative semiring. For any ideal $I$ of $A$ and $n\in \mathbb{N}$, let $I^{[n]}:=\{x\in A~|~x=\sum_{i=1}^s r_ia_i^n\text{~for some~} a_i\in I, r_i\in A \text{~and~}s\in \mathbb{N}\}$. Let $ \mathcal{I}$ denote the collection of all ideals of $A$. For any ideal $I\in \mathcal{I}$, 
\begin{equation}
I \mapsto I^{Frob} := \bigcup_{n \in \mathbb{N}}f_n^{-1}((f_n(I)A)) = \{x \in A~ \mid ~\exists~n \in \mathbb{N} \text{ such that } x^n \in I^{[n]}\}
\end{equation}
defines a closure operation on $\mathcal{I}$ which we call the Frobenius closure operation.\footnote{We call this the ``Frobenius closure" because it is motivated by the usual Frobenius closure for rings of characteristic $p > 0$. If $R$ is a ring of characteristic $p>0$, the association $a \mapsto a^{p^e}$ ($e \in \mathbb{N}$) defines a ring endomorphism. Clearly, this does not hold in general for arbitrary $n \in \mathbb{N}$. However, in an idempotent cancellative semiring, $a \mapsto a^n$ defines an endomorphism of $A$ for any $n$ and so we can consider Frobenius closure for any $n$.}
\end{pro}
\begin{proof}
Since $f_n$ is an endomorphism of $A$ for any $n \in \mathbb{N}$, it follows that $I^{Frob}$ is clearly an ideal of $A$. Extension and order-preservation properties are obvious. As for idempotence, let $x \in (I^{Frob})^{Frob}$. Then, there exists some $n \in \mathbb{N}$ such that $x^n =\sum_{i=1}^s r_ia_i^n$ where $a_i \in I^{Frob}$ and $r_i \in A$. For each $a_i$ there exists $k_i\in\mathbb{N}$ such that $a_i ^{k_i} \in I^{[k_i]}$. Let $t = \prod_{i=1}^s k_i$. Then, $x^{nt} = (\sum_{i=1}^s r_ia_i^n)^t = \sum_{i=1}^s (r_i)^ta_i^{nt}$ following Lemma \ref{binom trick} (2). Let $t_i =  \prod_{j=1, j\neq i}^s k_j$. Then, we have
\begin{equation}\label{sum}
x^{nt} = \sum_{i=1}^s (r_i)^ta_i^{nt} = \sum_{i=1}^s (r_i)^t(a_i^{k_i})^{nt_i}
\end{equation}
and since $a_i ^{k_i} \in I^{[k_i]}$ it follows from \eqref{sum} that $x^{nt} \in I^{[nt]}$ by applying Lemma \ref{binom trick} (2) again. Thus, $x\in I^{Frob}$. Since the other inclusion $I^{Frob} \subseteq (I^{Frob})^{Frob}$ is obvious, this completes the proof.
\end{proof}

\begin{rmk}
It is worth noting that one can define the Frobenius closure more generally for semirings that satisfy the following condition:
\begin{equation}\label{eq: Frobenius}
x^n+y^n=(x+y)^n, \quad \forall x,y \in A,~n \in \mathbb{N}.
\end{equation}
Such semirings include those whose trivial congruence is radical; this includes cancellative semirings since quotient cancellative congruences (as in \cite{joo2014prime}) are radical and totally ordered semirings. 

To see this why \eqref{eq: Frobenius} holds for these cases, for the case where the trivial congruence is radical, two elements are equal if and only if they are equal modulo each prime congruence, so it reduces to the cancellative case.  For totally ordered semirings, one can easily prove this by splitting into cases based on whether $x$ or $y$ is larger.	
\end{rmk}

\subsection{Closure operations for semirings via congruence relation}\label{Cl by cong}

In this section, we provide example of finite type closure operation on a set of congruences on a semiring $A$.

\begin{mydef}\label{definition: closure for congruence}
	Let $A$ be a semiring and $\mathcal{C}$ be a set of congruence relations on $A$. A closure operation $cl$ on $\mathcal{C}$ is a set map:
	\[
	cl:\mathcal{C} \to \mathcal{C}, \quad C\mapsto C^{cl}
	\] 	
	which satisfies the following:
	\begin{enumerate}
		\item 
		(Extension) $C \subseteq C^{cl}$ for all $C \in \mathcal{C}$. 
		\item 
		(Idempotence) $C^{cl}=(C^{c})^{cl}$ for all $C \in \mathcal{C}$. 
		\item 
		(Order-preservation) If $C_1 \subset C_2$, then $C_1^{cl} \subseteq C_2^{cl}$ for all $C_1,C_2 \in \mathcal{C}$. 
	\end{enumerate}
\end{mydef}

\begin{rmk}
	It follows from Remark \ref{remark: classical theory} that when $A$ is a ring, Definition \ref{definition: closure for congruence} is the same as a closure operation on a ring. 
\end{rmk}

\subsubsection{Radical closure}
One defines the \emph{radical} $\sqrt{C}$ of a congruence $C$ as the intersection of all prime congruences containing $C$.

\begin{pro}\label{proposition: radical closure for ideals}
Let $A$ be an idempotent semiring and $\mathcal C$ be the set of congruence relations on $A$. Then, \[
Rad: \mathcal{C}\longrightarrow \mathcal{C}, \quad C \mapsto \sqrt{C}
\]
is a closure operation. 
\end{pro}	
\begin{proof}
$(1)$ and $(3)$ of Definition \ref{definition: closure for congruence} are clear. For $(2)$, we only have to show that if $C$ is a congruence and $P$ is a prime congruence, then 
\[
C \subseteq P \iff \sqrt{C} \subseteq P. 
\]
If $P$ is a prime congruence, then one has $\sqrt{P}=P$ by \cite[Proposition 3.12]{joo2014prime}. Therefore, if $C \subseteq P$ then we have that $\sqrt{C} \subseteq \sqrt{P}=P$. The converse is also clear since $C \subseteq \sqrt{C}$.	
\end{proof}

One may also see the idempotence of the radical closure operation through \cite[Theorem 3.9]{joo2014prime} which gives another description of $\sqrt{C}$. In what follows, all products are twisted product unless otherwise stated. 

\begin{mydef}\cite[Definition 3.4]{joo2014prime}
Let $A$ be an idempotent semiring. For an element $\alpha=(x,y) \in A$, a generalized power of $\alpha$ is an element of $A\times A$ of the following from:
\[
((\alpha^*)^m +(c,0))\alpha^n, \quad m,n \in \mathbb{N},c \in A,
\]	
where $\alpha^*:=(x+y,0)$. We let $GP(\alpha)$ be the set of all generalized powers of $\alpha$. 
\end{mydef}

Jo\'o and Mincheva proved the following:

\begin{mytheorem}\cite[Theorem 3.9]{joo2014prime}
Let $C$ be a congruence on an idempotent semiring. One has the following:
\begin{equation}\label{equation: radical equation}
\sqrt{C}= \{ \alpha \mid GP(\alpha) \cap C \neq \emptyset\}.
\end{equation}
\end{mytheorem}	

We have the following proof showing that $\sqrt{\sqrt{C}}=\sqrt{C}$ only by using generalized powers. 

\begin{pro}
Let $A$ be an idempotent semiring and $C$ be a congruence on $A$, then we have $$\sqrt{\sqrt{C}}=\sqrt{C}.$$ 
\end{pro}
\begin{proof}
It is clear that $\sqrt{C} \subseteq \sqrt{\sqrt{C}}$. Suppose that $\alpha=(\alpha_1,\alpha_2) \in \sqrt{\sqrt{C}}$, that is there exist $i,j \in \mathbb{N}$ and $c \in A$ such that 
\begin{equation}\label{equation1}
\beta:=(\beta_1,\beta_2) = ( (\alpha_1+\alpha_2)^i +c, 0) \alpha^j \in \sqrt{C}. 
\end{equation} 
Then, since $\beta \in \sqrt{C}$, we further have $i',j' \in \mathbb{N}$, $c' \in A$ such that
\begin{equation}\label{equation2}
( (\beta_1+\beta_2)^{i'}+c',0)\beta^{j'} \in C. 
\end{equation}
By substituting $\beta$ in \eqref{equation1}, we have that
\begin{equation}\label{equation3}
( (\beta_1+\beta_2)^{i'}+c',0) ( (\alpha_1+\alpha_2)^i +c,0)^{j'}\alpha^{jj'} \in C. 
\end{equation}
Since $(a,0)^n = (a^n,0)$ for any $a \in A$, we can write \eqref{equation3} as follows:
\begin{equation}\label{equation4}
( (\beta_1+\beta_2)^{i'}+c',0) ( (\alpha_1+\alpha_2)^{ij'} +k,0)\alpha^{jj'} \in C, \textrm{ for some } k \in A. 
\end{equation}
We claim the following:
\begin{equation}\label{claim}
(\beta_1+\beta_2)^{i'} +c' = (\alpha_1+\alpha_2)^{ii'j} +c'', \textrm{ for some }
 c'' \in A.
\end{equation}
If our claim holds, then it follows from \eqref{equation4} that
\[
( (\alpha_1+\alpha_2)^{ii'j}+c'',0)( (\alpha_1+\alpha_2)^{ij'}+k,0)\alpha^{ij} =( (\alpha_1+\alpha_2)^{ii'j+ij'}+k',0)\alpha^{jj'}, \textrm{ for some } k' \in A,
\]
showing that $\alpha \in \sqrt{C}$. 

We now only have to prove our claim. Let $\alpha^j:=(\gamma_1,\gamma_2)$. One can easily observe the following:
\[
\gamma_1+ \gamma_2 = (\alpha_1+\alpha_2)^j.
\]
Then, we can rewrite $\beta$ in terms of $\gamma_1$ and $\gamma_2$:
\[
\beta=(\beta_1,\beta_2)=( (\alpha_1+\alpha_2)^i +c, 0) (\gamma_1,\gamma_2) = ( ((\alpha_1+\alpha_2)^i +c)\gamma_1, ((\alpha_1+\alpha_2)^i +c)\gamma_2  )
\]
We have that
\[
\beta_1+\beta_2 = ((\alpha_1+\alpha_2)^i +c) (\gamma_1+\gamma_2)=((\alpha_1+\alpha_2)^i +c)(\alpha_1+\alpha_2)^j
\]
\[
=(\alpha_1+\alpha_2)^{ij} + a, \textrm{ for some } a \in A, 
\]
and hence we have
\[
(\beta_1+\beta_2)^{i'} = (\alpha_1+\alpha_2)^{ii'j} + a', \textrm{ for some } a' \in A. 
\]
This proves our claim by adding $c'$ to both hand sides. 
\end{proof}

\begin{rmk}
The radical closure operation on congruences is indeed a finite type closure operation. This can be easily verified by using the description of radical closure through generalized powers.
\end{rmk}

\bibliography{spectral}\bibliographystyle{alpha}

\begin{thebibliography}{IKR11b}

\bibitem[BE17]{bertram2017tropical}
Aaron Bertram and Robert Easton.
\newblock The tropical nullstellensatz for congruences.
\newblock {\em Advances in Mathematics}, 308:36--82, 2017.

\bibitem[BG16]{borger2016boolean}
James Borger and Darij Grinberg.
\newblock Boolean {W}itt vectors and an integral {E}drei--{T}homa theorem.
\newblock {\em Selecta Mathematica}, 22(2):595--629, 2016.

\bibitem[Bor16]{borger2016witt}
James Borger.
\newblock Witt vectors, semirings, and total positivity.
\newblock {\em Absolute Arithmetic and $\mathbb{F}_1$-Geometry}, pages
  273--329, 2016.

\bibitem[CC09]{con5}
Alain Connes and Caterina Consani.
\newblock Characteristic 1, entropy and the absolute point.
\newblock In {\em Noncommutative Geometry, Arithmetic, and Related Topics,
  Proceedings of the 21st Meeting of the Japan-US Mathematics Institute,
  Baltimore}, pages 75--140, 2009.

\bibitem[CC19]{connes2017homological}
Alain Connes and Caterina Consani.
\newblock Homological algebra in characteristic one.
\newblock {\em Higher Structures Journal}, 3(1):155--247, 2019.

\bibitem[Cul19]{culling}
Robert Hendrik~Scott Culling.
\newblock On the \'{E}tale fundamental group of schemes over the natural
  numbers.
\newblock {\em Ph.D. thesis, Australian National University}, 2019.

\bibitem[Dud17]{dudzik2017quantales}
Andrew Dudzik.
\newblock Quantales and hyperstructures: Monads, mo'problems.
\newblock {\em arXiv preprint arXiv:1707.09227}, 2017.

\bibitem[Eps12]{epstein2012guide}
Neil Epstein.
\newblock A guide to closure operations in commutative algebra.
\newblock {\em Progress in commutative algebra}, 2:1--37, 2012.

\bibitem[FFS16]{Fo}
Carmelo~A Finocchiaro, Marco Fontana, and Dario Spirito.
\newblock A topological version of {H}ilbert's {N}ullstellensatz.
\newblock {\em Journal of algebra}, 461:25--41, 2016.

\bibitem[Fin14]{F}
Carmelo~Antonio Finocchiaro.
\newblock Spectral spaces and ultrafilters.
\newblock {\em Communications in Algebra}, 42(4):1496--1508, 2014.

\bibitem[GG16]{giansiracusa2016equations}
Jeffrey Giansiracusa and Noah Giansiracusa.
\newblock Equations of tropical varieties.
\newblock {\em Duke Mathematical Journal}, 165(18):3379--3433, 2016.

\bibitem[Hoc69]{hochster1969prime}
Melvin Hochster.
\newblock Prime ideal structure in commutative rings.
\newblock {\em Transactions of the American Mathematical Society}, 142:43--60,
  1969.

\bibitem[HS06]{HS}
Craig Huneke and Irena Swanson.
\newblock {\em Integral closure of ideals, rings, and modules}, volume~13.
\newblock Cambridge University Press, 2006.

\bibitem[IKR11a]{izhakian2011glimpse}
Zur Izhakian, Manfred Knebusch, and Louis Rowen.
\newblock A glimpse at supertropical valuation theory.
\newblock {\em St. Univ. Ovidius Constanta}, 19:131--142, 2011.

\bibitem[IKR11b]{izhakian2011supertropical}
Zur Izhakian, Manfred Knebusch, and Louis Rowen.
\newblock Supertropical semirings and supervaluations.
\newblock {\em Journal of Pure and Applied Algebra}, 215(10):2431--2463, 2011.

\bibitem[Jac80]{j2}
N.~Jacobson.
\newblock {\em Basic Algebra II}.
\newblock W.H. Freeman and Company, 1980.

\bibitem[Jec13]{Je}
Thomas Jech.
\newblock {\em Set theory}.
\newblock Springer Science \& Business Media, 2013.

\bibitem[JM18]{joo2014prime}
D{\'a}niel Jo{\'o} and Kalina Mincheva.
\newblock Prime congruences of additively idempotent semirings and a
  nullstellensatz for tropical polynomials.
\newblock {\em Selecta Mathematica}, 24(3):2207--2233, 2018.

\bibitem[JMR19]{jun2019projective}
Jaiung Jun, Kalina Mincheva, and Louis Rowen.
\newblock Projective systemic modules.
\newblock {\em Journal of Pure and Applied Algebra}, page 106243, 2019.

\bibitem[Jun17]{jun2017vcech}
Jaiung Jun.
\newblock {\v{C}}ech cohomology of semiring schemes.
\newblock {\em Journal of Algebra}, 483:306--328, 2017.

\bibitem[Jun18]{jun2018valuations}
Jaiung Jun.
\newblock Valuations of semirings.
\newblock {\em Journal of Pure and Applied Algebra}, 222(8):2063--2088, 2018.

\bibitem[Les11]{les2}
Paul Lescot.
\newblock Absolute algebra {II}—{I}deals and spectra.
\newblock {\em Journal of Pure and Applied Algebra}, 215(7):1782--1790, 2011.

\bibitem[Les12]{les3}
Paul Lescot.
\newblock Absolute algebra {III}—the saturated spectrum.
\newblock {\em Journal of Pure and Applied Algebra}, 216(5):1004--1015, 2012.

\bibitem[Les15]{PP}
Paul Lescot.
\newblock {Prime and primary ideals in semirings}.
\newblock {\em Osaka Journal of Mathematics}, 52(3):721 -- 737, 2015.

\bibitem[Lor19]{lorscheid2019tropical}
Oliver Lorscheid.
\newblock Tropical geometry over the tropical hyperfield.
\newblock {\em arXiv preprint arXiv:1907.01037}, 2019.

\bibitem[Min16]{mincheva2016semiring}
Kalina Mincheva.
\newblock {\em Semiring Congruences and Tropical Geometry}.
\newblock PhD thesis, Johns Hopkins University, 2016.

\bibitem[Mor]{adicspace}
Sophie Morel.
\newblock Adic spaces (lecture note).
\newblock \url{https://web.math.princeton.edu/~smorel/adic_notes.pdf}.
\newblock Accessed: 2019-12-29.

\bibitem[MR18]{maclagan2016tropical}
Diane Maclagan and Felipe Rinc{\'o}n.
\newblock Tropical ideals.
\newblock {\em Compositio Mathematica}, 154(3):640--670, 2018.

\bibitem[MS09]{maclagan2009introduction}
Diane Maclagan and Bernd Sturmfels.
\newblock Introduction to tropical geometry.
\newblock {\em Graduate Studies in Mathematics}, 161:75--91, 2009.

\bibitem[Ray19]{ray2018closure}
Samarpita Ray.
\newblock Closure operations, continuous valuations on monoids and spectral
  spaces.
\newblock {\em Journal of Algebra and Its Applications, doi:
  10.1142/S0219498820500061.}, 2019.

\bibitem[SA92]{senadhikari1992}
M.K Sen and M.R. Adhikari.
\newblock On $k$-ideals of semirings.
\newblock {\em International Journal of Mathematics and Mathematical Sciences},
  15:Article ID 642431, 4 pages, 1992.

\bibitem[Ste10]{steinberg2010lattice}
Stuart~A Steinberg.
\newblock {\em Lattice-ordered rings and modules}.
\newblock Springer, 2010.

\bibitem[Tak10]{takagi2010construction}
Satoshi Takagi.
\newblock Construction of schemes over $\mathbb{F}_1$, and over idempotent
  semirings: towards tropical geometry.
\newblock {\em arXiv preprint arXiv:1009.0121}, 2010.

\bibitem[Tol16]{tolliver2016extension}
Jeffrey Tolliver.
\newblock Extension of valuations in characteristic one.
\newblock {\em arXiv preprint arXiv:1605.06425}, 2016.

\bibitem[Vir10]{viro}
Oleg Viro.
\newblock Hyperfields for tropical geometry {I}. hyperfields and
  dequantization.
\newblock {\em arXiv preprint arXiv:1006.3034}, 2010.

\bibitem[Wed]{adicspace2}
Torsten Wedhorn.
\newblock Adic spaces (lecture note).
\newblock
  \url{http://math.stanford.edu/~conrad/Perfseminar/refs/wedhornadic.pdf}.
\newblock Accessed: 2019-12-29.

\end{thebibliography}

\end{document}